\UseRawInputEncoding

\documentclass[11pt]{amsart}

\usepackage{amsfonts,epsfig}
\usepackage{latexsym}
\usepackage{amssymb}
\usepackage{amsmath}
\usepackage{amsthm}
\usepackage{graphics}
\usepackage[all]{xy}
\usepackage[T2A]{fontenc}
\usepackage{multirow}
\usepackage{array, booktabs, ctable}
\usepackage[pagebackref]{hyperref}
\usepackage{enumerate}

\newtheorem{defn}{Definition}[section]

\newtheorem{corollary}[defn]{Corollary}
\newtheorem{lemma}[defn]{Lemma}

\newtheorem{thm}[defn]{Theorem}

\newtheorem{cor}[defn]{Corollary}
\newtheorem{prop}[defn]{Proposition}

\newtheorem{conj}[defn]{Conjecture}

\theoremstyle{definition}

\newtheorem*{ack}{Acknowledgements}

\newcommand{\CC}{\mathbb C}

\newcommand{\Q}{\mathbb Q}
\newcommand{\Z}{\mathbb Z}

\newcommand{\F}{\mathbb F}

\newcommand{\PP}{\mathbb P}

\newcommand{\Rats}{\mathbb{Q}}

\newcommand{\Gal}{\operatorname{Gal}}
\newcommand{\Aut}{\operatorname{Aut}}
\newcommand{\GQ}{\Gal(\overline{\Rats}/\Rats)}

\newcommand{\GL}{\operatorname{GL}}
\newcommand{\PGL}{\operatorname{PGL}}

\newcommand{\spl}{\mathcal{C}_{\text{sp}}}

\begin{document}

\bibliographystyle{plain}
\title[Abelian division fields]{Elliptic Curves with abelian division fields}

\author{Enrique Gonz\'alez--Jim\'enez}
\address{Universidad Aut{\'o}noma de Madrid, Departamento de Matem{\'a}ticas, Madrid, Spain}
\email{enrique.gonzalez.jimenez@uam.es}
\author{\'Alvaro Lozano-Robledo}
\address{University of Connecticut, Department of Mathematics, Storrs, CT 06269, USA}
\email{alvaro.lozano-robledo@uconn.edu}

\thanks{The first author was partially  supported by the grant MTM2012--35849.}

\begin{abstract} Let $E$ be an elliptic curve over $\Q$, and let $n\geq 1$. The central object of study of this article is the  division field $\Q(E[n])$ that results by adjoining to $\Q$ the coordinates of all $n$-torsion points on $E(\overline{\Q})$. In particular, we classify all curves $E/\Q$ such that $\Q(E[n])$ is as small as possible, that is, when $\Q(E[n])=\Q(\zeta_n)$, and we prove that this is only possible for $n=2,3,4$, or $5$. More generally, we classify all curves such that $\Q(E[n])$ is contained in a cyclotomic extension of $\Q$ or, equivalently (by the Kronecker-Weber theorem), when $\Q(E[n])/\Q$ is an abelian extension. In particular, we prove that this only happens for $n=2,3,4,5,6$, or $8$, and we classify the possible Galois groups that occur for each value of $n$.
\end{abstract}

\maketitle

\section{Introduction}

Let $E$ be an elliptic curve defined over $\Q$, and let $n\geq 1$. The central object of study of this article is the division field $\Q(E[n])$ that results by adjoining to $\Q$ the coordinates of all $n$-torsion points on $E(\overline{\Q})$, where $\overline{\Q}$ is a fixed algebraic closure of $\Q$. The existence of the Weil pairing (see \cite[III, Corollary 8.1.1]{silverman}) implies that $\Q(E[n])$ contains all the $n$-th roots of unity of $\overline{\Q}$, i.e.,~we have an inclusion $\Q(\zeta_{n})\subseteq \Q(E[n])$, where $\zeta_{n}$ is a primitive $n$-th root of unity. The goal of this article is to study the case when $\Q(E[n])$ is as small as possible, that is, when $\Q(E[n])=\Q(\zeta_n)$ and, more generally, when $\Q(E[n])$ is contained in a cyclotomic extension of $\Q$ or, equivalently (by the Kronecker-Weber theorem), when $\Q(E[n])/\Q$ is an abelian extension.  For instance:
\begin{itemize}
\item ($n=2$)  $E_{15a2}: y^2+xy+y=x^3+x^2-135x-660$, satisfies $\Q(E[2])=\Q=\Q(\zeta_2)$,
\item ($n=3$)  $E_{19a1}: y^2 + y = x^3 + x^2 - 9x - 15$, satisfies $\ \Q(E[3])=\Q(\sqrt{-3})=\Q(\zeta_3)$,
\item ($n=4$) $E_{15a1}:y^2 +xy+y= x^3 +x^2 -10x - 10$, satisfies $\Q(E[4])=\Q(i)=\Q(\zeta_{4})$,
\item ($n=5$) $E_{11a1}:y^2 + y = x^3 - x^2 - 10x - 20$, satisfies $\Q(E[5])=\Q(\zeta_5)$,
\item ($n=6$) $E_{14a1}:y^2+xy+y=x^3+4x-6$, satisfies $\Q(E[6])=\Q(\zeta_6,\sqrt{-7})$,
\item ($n=8$) $E_{15a1}:y^2 +xy+y= x^3 +x^2 -10x - 10$, satisfies $\Q(E[8])=\Q(\zeta_{8},\sqrt{3},\sqrt{7})$.
\end{itemize}

 Previously, Paladino \cite{paladino} has classified all the curves $E/\Q$ with $\Q(E[3])=\Q(\zeta_3)$. In a more general setting, when $E$ is defined over a number field $K$, the work of Halberstadt, Merel \cite{merel}, Merel and Stein \cite{merel2}, and Rebolledo \cite{rebolledo} shows that if $p$ is prime, and $K(E[p])=\Q(\zeta_p)$, then $p=2,3,5$ or $p>1000$. 

The main result of this article is a complete classification (and parametrization) of all elliptic curves $E/\Q$, up to isomorphism over $\Q$, such that $\Q(E[n])$ is abelian over $\Q$, and those curves such that $\Q(E[n])=\Q(\zeta_n)$. Furthermore, we classify all the abelian Galois groups $\Gal(\Q(E[n])/\Q)$ for each value of $n\geq 2$ that may occur. 

\

\begin{thm}\label{thm-main} Let $E/\Q$ be an elliptic curve. If there is an integer $n\geq 2$  such that $\Q(E[n])=\Q(\zeta_n)$, then $n=2,3,4,$ or $5$. More generally, if $\Q(E[n])/\Q$ is abelian, then  $n=2,3,4,5,6$, or $8$.  Moreover, $\Gal(\Q(E[n])/\Q)$ is isomorphic to one of the following groups:
\begin{center}
\renewcommand{\arraystretch}{1.2}
\begin{tabular}{|c||c|c|c|c|c|c|}
\hline
$n$ & $2$ & $3$ & $4$ & $5$ & $6$ & $8$\\
\hline 
\multirow{4}{*}{$\Gal(\Q(E[n])/\Q)$}  & $\{0\}$ & $\Z/2\Z$ & $\Z/2\Z$ & $\Z/4\Z$ & $(\Z/2\Z)^2$ & $(\Z/2\Z)^4$ \\
 & $\Z/2\Z$ & $(\Z/2\Z)^2$ & $(\Z/2\Z)^2$ & $\Z/2\Z\times \Z/4\Z$ & $(\Z/2\Z)^3$ & $(\Z/2\Z)^5$ \\
 & $\Z/3\Z$ &  & $(\Z/2\Z)^3$ & $(\Z/4\Z)^2$ & & $(\Z/2\Z)^6$\\
 &  &  & $(\Z/2\Z)^4$ & & & \\
 \hline
\end{tabular}
\end{center}
Furthermore, each possible Galois group occurs for infinitely many distinct $j$-invariants.
\end{thm} 

Let $\rho_{E,n}\colon\Gal(\overline{\Q}/\Q)\to \Aut(E[n])\cong \GL(2,\Z/n\Z)$ be the representation induced by the action of Galois on $E[n]$. In \cite[\S 4.3]{serre1}, Serre asked whether there is a constant $N$, that does not depend on $E$, and such that $\rho_{E,p}$ is surjective for all elliptic curves $E/\Q$ without CM, and for all $p> N$. Serre actually asks whether $N=37$ works. This question, usually known as ``Serre's uniformity problem'', has generated great interest (see \cite{bilu}, \cite{alina1}, \cite{alina2}, \cite{daniels}, \cite{kraus}, \cite{masser}, \cite{pellarin}). We note here that, as a corollary of our theorem, we obtain the following uniform statement: 

\begin{corollary}
For any $n\geq 9$, and any elliptic curve $E/\Q$, the image of $\rho_{E,n}$ is non-abelian.
\end{corollary}

In the course of proving Theorem \ref{thm-main}, we show the analogous statement in the simpler case of elliptic curves with complex multiplication.

\begin{thm}\label{thm-maincm0} Let $E/\Q$ be an elliptic curve with complex multiplication by an imaginary quadratic field $K$, with discriminant $d_K$. If there is an integer $n\geq 2$  such that $\Q(E[n])=\Q(\zeta_n)$, then $n=2$, or $3$. More generally, if $\Q(E[n])/\Q$ is abelian, then  $n=2,3$, or $4$.  Moreover, $\Gal(\Q(E[n])/\Q)$ is isomorphic to one of the following groups:
\begin{center}
\renewcommand{\arraystretch}{1.3}
\begin{tabular}{|c||c|c|c|c|}
\hline
$n$ & \multicolumn{2}{c|}{$2$}  & $3$ & $4$ \\
\hline
$d_K$ & $-4$ & $-3,-7,-8$  & $-3$ & $-4$ \\
\hline 
\multirow{2}{*}{$\Gal(\Q(E[n])/\Q)$}  & $\{0\}$ & $\Z/2\Z$  & $\Z/2\Z$ & $(\Z/2\Z)^2$  \\
  &  $\Z/2\Z$ &  & $(\Z/2\Z)^2$ & $(\Z/2\Z)^3$  \\
\hline
\end{tabular}
\end{center}
\end{thm}

The structure of the article is as follows. In Section \ref{sec-prelim} we quote useful prior results in the literature about torsion subgroups and isogenies, which we use in Section \ref{sec-someproofs} to prove a number of intermediary results. In Section \ref{sec-cm}, we study the problem of abelian division fields for the subset of elliptic curves over $\Q$ with complex multiplication, and we prove Theorem \ref{thm-maincm0} (see Theorem \ref{thm-maincm}). In Section \ref{sec-mainproof} we prove Theorem \ref{thm-main} (see Theorems \ref{thm-main1} and \ref{thm-main2}), and in Section \ref{sec-param} we describe the (infinite) families with each combination of abelian division field and torsion structure over $\Q$. Finally, in the Appendix (a.k.a.~Section \ref{sec-appendix}) we include tables with concrete examples of elliptic curves with each possible combination of Galois group and rational torsion structure.

\begin{ack}
This project began after a conversation with J. M. Tornero about torsion points over cyclotomic fields, so we would like to thank him for the initial questions that inspired this work. We would also like to thank David Zywina for providing us with several models of modular curves. The second author would like to thank the Universidad Aut\'onoma de Madrid, where much of this work was completed during a sabbatical visit, for its hospitality. The authors would also like to thank the editors and the referee for their comments and suggestions. Finally, the authors would like to thank Tyler Genao for pointing out an error in an earlier version of the proof of Proposition \ref{prop-nonsplit}, which is fixed in this version.
\end{ack}

\section{Prior literature}\label{sec-prelim}

In this section we cover a number of results on torsion subgroups and isogenies that will be needed in the proof of Theorem \ref{thm-main}. First, we quote two results on torsion subgroups over quadratic fields.

\begin{thm}[Kenku, Momose, \cite{kenku2}; Kamienny, \cite{kami}]\label{thm-quadgroups}
Let $K/\Q$ be a quadratic field and let $E/K$ be an elliptic
curve. Then
\[
E(K)_\text{tors}\simeq
\begin{cases}
\Z/M\Z &\text{with}\ 1\leq M\leq 16\ \text{or}\ M=18,\ \text{or}\\
\Z/2\Z\oplus \Z/2M\Z &\text{with}\ 1\leq M\leq 6,\ \text{or}\\
\Z/3\Z \oplus \Z/3M\Z &\text{with}\ M=1\ \text{or}\ 2,\ \text{only if}\ K = \Q(\sqrt{-3}),\ \text{or}\\
\Z/4\Z \oplus \Z/4\Z &\text{only if}\ K = \Q(\sqrt{-1}).
\end{cases}
\]
\end{thm}

\begin{thm}[Najman, \cite{najman}, Theorem 2]\label{thm-najman2}
\label{thm-najman1} Let $E/\Q$ be an elliptic curve defined over $\Q$, and let $F$ be a quadratic number field. Then,
\[
E(F)_\text{tors}\simeq
\begin{cases}
\Z/M\Z &\text{with}\ 1\leq M\leq 10, \text{ or } M=12,15,\ \text{or}\ 16, \text{ or}\\
\Z/2\Z\oplus \Z/2M\Z &\text{with}\ 1\leq M\leq 6,\ \text{or}\\
\Z/3\Z \oplus \Z/3M\Z &\text{with}\ M=1\ \text{or}\ 2,\ \text{only if}\ F = \Q(\sqrt{-3}),\ \text{or}\\
\Z/4\Z \oplus \Z/4\Z &\text{only if}\ F = \Q(\sqrt{-1}).
\end{cases}
\]
\end{thm}

The work of Laska, Lorenz, and Fujita, classifies all the possible torsion structures that can occur for elliptic curves over $\Q$ when base-changed to a polyquadratic number field (i.e.,~a number field contained in the maximal elementary $2$-abelian extension of $\Q$, which we will call $\Q(2^\infty)$ below). We quote their results next.

\begin{thm}[Laska, Lorenz, \cite{laska}; Fujita, \cite{fujita}]\label{thm-fujita} Let $E/\Q$ be an elliptic curve defined over $\Q$, and let 
$$\Q(2^\infty) = \Q(\{\sqrt{m}: m\in\Z\}).$$
Then, the torsion subgroup $E(\Q(2^\infty))_\text{tors}$ is finite, and it is isomorphic to one of the following groups:
\[
E(\Q(2^\infty))_\text{tors}\simeq
\begin{cases}
\Z/M\Z, &\text{with}\ M=1,3,5,7,9,\text{ or } M=15, \text{ or}\\
\Z/2\Z\oplus \Z/2M\Z &\text{with}\ 1\leq M\leq 6,\text{ or } M=8, \text{ or}\\
\Z/3\Z\oplus \Z/3\Z & \text{ or}\\
\Z/4\Z\oplus \Z/4M\Z &\text{with}\ 1\leq M\leq 4,\text{ or}\\
\Z/2M\Z\oplus \Z/2M\Z &\text{with}\ 3\leq M\leq 4.\\
\end{cases}
\]
\end{thm}

The following result is a criterion to decide whether a given point is twice another point over the same field (see \cite[Theorem 4.2]{knapp}).

\begin{prop}\label{prop-knapp} Let $E/K$ be an elliptic curve defined over a number field $K$, given by 
$$E:y^2 = (x-\alpha)(x-\beta)(x-\gamma)$$
with pairwise distinct $\alpha,\beta,\gamma\in K$. For $P=(x_0,y_0)\in E(K)$, there exists $Q\in E(K)$ such that $2Q=P$ if and only if $x_0-\alpha$, $x_0-\beta$, and $x_0-\gamma$ are all squares in $K$. 
\end{prop}
As a corollary of the previous criterion, we deduce a description of the $4$-torsion of an elliptic curve.
\begin{cor}\label{cor-E4}
 Let $E/F$ be an elliptic curve defined over a number field $F$, given by 
$$E:y^2 = (x-\alpha)(x-\beta)(x-\gamma)$$
with $\alpha,\beta,\gamma\in \overline{F}$. Then,
$$F(E[4])=F\left(\sqrt{\pm(\alpha-\beta)},\sqrt{\pm(\alpha-\gamma)},\sqrt{\pm(\beta-\gamma)}\right)=F\left(\sqrt{-1},\sqrt{\alpha-\beta},\sqrt{\alpha-\gamma},\sqrt{\beta-\gamma}\right).$$
\end{cor}
\begin{proof}
This follows directly from Proposition \ref{prop-knapp}, and the fact that $$2E[4]=E[2]=\{\mathcal{O},(\alpha,0),(\beta,0),(\gamma,0)\}.
$$
\end{proof}

We will also use the following result of Kenku on the maximum number of isogenous curves over $\Q$ in an isogeny class. The notation $C(E)$ and $C_p(E)$ introduced in the statement below will be used several times in our proofs.

\begin{thm}[Kenku, \cite{kenku}]\label{thm-kenku} There are at most eight $\Q$-isomorphism classes of elliptic curves in each $\Q$-isogeny class. More concretely, let $E/\Q$ be an elliptic curve, define $C(E)$ as the number of distinct $\Q$-rational cyclic subgroups of $E$ (including the identity subgroup), and let $C_p(E)$ be the same as $C(E)$ but only counting cyclic subgroups of order a power of $p$, for each prime $p$. Then, $C(E)=\prod_p C_p(E)\leq 8$. Moreover, each factor $C_p(E)$ is bounded as follows:
\begin{center}
\begin{tabular}{c|cccccccccccc}
$p$ & $2$ & $3$ & $5$ & $7$ & $11$ & $13$ & $17$ & $19$ & $37$ & $43$ & $67$ & $163$\\
\hline 
$C_p\leq $ & $8$ & $4$ & $3$ & $2$ & $2$ & $2$ & $2$ & $2$ & $2$ & $2$ & $2$ & $2$.
\end{tabular}
\end{center}
\end{thm}

The next result we quote describes the possible isomorphism types of $\Gal(\Q(E[p])/\Q)$. In particular, fix a $\Z/p\Z$-basis of $E[p]$, and let $\rho_{E,p}\colon \Gal(\overline{\Q}/\Q)\to \GL(E[p])\cong \GL(2,\Z/p\Z)$ be the representation associated to the natural action of Galois on $E[p]$, with respect to the chosen basis of $E[p]$. Then, $\Gal(\Q(E[p])/\Q)\cong \rho_{E,p}(\GQ)\subseteq \GL(E[p])\cong \GL(2,\Z/p\Z)$.

\begin{thm}[Serre, \cite{serre1}]\label{thm-serre2}
Let $E/\Q$ be an elliptic curve. Let $G$ be the image of $\rho_{E,p}$, and suppose $G\neq \GL(E[p])$. Then, there is a $\Z/p\Z$-basis of $E[p]$ such that one of the following possibilities holds:
\begin{enumerate}
\item $G$ is contained in the normalizer of a split Cartan subgroup of $\GL(E[p])$, or 
\item $G$ is contained in the normalizer of a non-split Cartan subgroup of $\GL(E[p])$, or 
\item The projective image of $G$ in $\PGL(E[p])$ is isomorphic to $A_4$, $S_4$ or $A_5$, where $S_n$ is the symmetric group and $A_n$ the alternating group, or 
\item $G$ is contained in a Borel subgroup of $\GL(E[p])$.
\end{enumerate}
\end{thm}

Rouse and Zureick-Brown have classified all the possible $2$-adic images of $\rho_{E,2}\colon \GQ\to \GL(2,\Z_2)$ (see previous results of Dokchitser and Dokchitser \cite{dok} on the surjectivity of $\rho_{E,2} \bmod 2^n$), and Sutherland and Zywina have conjectured the possibilities for the mod $p$ image for all primes $p$.

\begin{thm}[Rouse, Zureick-Brown, \cite{rouse}]\label{thm-rzb} Let $E$ be an elliptic curve over $\Q$ without complex multiplication. Then, there are exactly $1208$ possibilities for the $2$-adic image $\rho_{E,2^\infty}(\GQ)$, up to conjugacy in $\GL(2,\Z_2)$. Moreover, the index of $\rho_{E,2^\infty}(\Gal(\overline{\Q}/\Q))$ in $\GL(2,\Z_2)$ divides $64$ or $96$.
\end{thm}

\begin{conj}[Sutherland, Zywina, \cite{zywina1}]\label{conj-zyw} Let $E/\Q$ be an elliptic curve. Let $G$ be the image of $\rho_{E,p}$. Then, there are precisely $63$ isomorphism types of images. 
\end{conj}

The $\Q$-rational points on the modular curves $X_0(N)$ have been described completely in the literature, for all $N\geq 1$. One of the most important milestones in their classification was \cite{mazur1}, where Mazur dealt with the case when $N$ is prime. The complete classification of $\Q$-rational points on $X_0(N)$, for any $N$, was completed due to work of Fricke, Kenku, Klein, Kubert, Ligozat, Mazur  and Ogg, among others (see \cite[eq. (80)]{elkies1}; \cite{elkies2}; \cite{kleinfricke}, \cite[pp. 370-458]{fricke}; \cite[p. 1889]{ishii}; \cite{maier}; \cite{lnm476}; \cite{mazur1}; \cite{kenku}; or the summary tables in \cite{lozano0}).

\begin{thm}\label{thm-ratnoncusps} Let $N\geq 2$ be a number such that $X_0(N)$ has a non-cuspidal $\Q$-rational point. Then:
\begin{enumerate}
\item $N\leq 10$, or $N= 12,13, 16,18$ or $25$. In this case $X_0(N)$ is a curve of genus $0$ and its $\Q$-rational points form an infinite $1$-parameter family, or
\item $N=11,14,15,17,19,21$, or $27$. In this case $X_0(N)$ is a curve of genus $1$, i.e.,~$X_0(N)$ is an elliptic curve over $\Q$, but in all cases the Mordell-Weil group $X_0(N)(\Q)$ is finite, or 

\item $N=37,43,67$ or $163$. In this case $X_0(N)$ is a curve of genus $\geq 2$ and (by Faltings' theorem) there are only finitely many $\Q$-rational points.
\end{enumerate} 
\end{thm}

\section{Preliminary results}\label{sec-someproofs}
In this section we prove a number of results that will be used in Section \ref{sec-mainproof} to prove the main results of this paper.

\begin{defn}\label{defn-borel1} Let $p$ be a prime, and $n\geq 1$. We say that a subgroup $B$ of $\GL(2,\Z/p^n\Z)$ is Borel if every matrix in $B$ is upper triangular, i.e.,
$$B\leq \left\{ \left(\begin{array}{cc} a & b \\ 0 & c\end{array}\right) : a,b,c\in\Z/p^n\Z,\ a,c\in(\Z/p^n\Z)^\times \right\}.$$
We say that $B$ is a non-diagonal Borel subgroup if none of the conjugates of $B$ in $\GL(2,\Z/p^n\Z)$ are formed solely by diagonal matrices. If $B$ is a Borel subgroup, we denote by $B_1$ the subgroup of $B$ formed by those matrices in $B$ whose diagonal coordinates are $1\bmod p^n$, and we denote by $B_d$ the subgroup of $B$ formed by diagonal matrices, i.e.,
$$B_1 = B\cap \left\{ \left(\begin{array}{cc} 1 & b \\ 0 & 1\end{array}\right) : b\in\Z/p^n\Z \right\}, \text{ and } B_d = B\cap \left\{ \left(\begin{array}{cc} a & 0 \\ 0 & c\end{array}\right) : a,c\in(\Z/p^n\Z)^\times \right\}.$$
\end{defn}
\begin{lemma}[\cite{lozano1}, Lemma 2.2]\label{lem-borel} Let $p>2$ be a prime, $n\geq 1$ and let $B\leq \GL(2,\Z/p^n\Z)$ be a Borel subgroup, such that $B$ contains a matrix $g=\left(\begin{array}{cc} a & b \\ 0 & c\end{array}\right)$ with $a\not\equiv c \bmod p$. Let $B'=h^{-1}Bh$ with $h=\left(\begin{array}{cc} 1 & b/(c-a) \\ 0 & 1\end{array}\right)$. Then, $B'\leq \GL(2,\Z/p^n\Z)$ is a Borel subgroup conjugate to $B$ satisfying the following properties:
\begin{enumerate}
\item $B' = B_d'B_1'$, i.e.,~for every $M\in B'$ there is $U\in B_d'$ and $V\in B_1'$ such that $M=UV$, and
\item $B/[B,B]\cong B'/[B',B']$ and $[B',B']=B_1'$. If follows that $[B,B]=B_1$ and it is a cyclic subgroup of order $p^s$ for some $0\leq s\leq n$.
\end{enumerate}
\end{lemma}

\begin{lemma}\label{lem-borel3}
Suppose $p>2$ and $G\subseteq \GL(2,\Z/p^2\Z)$ is an abelian subgroup that is diagonal when reduced modulo $p$, and such that $G$  contains a matrix $g=\left(\begin{array}{cc} a & bp \\ cp & d\end{array}\right)$ with $a\not\equiv d \bmod p$. Then, $G$ is a conjugate of a diagonal subgroup of $\GL(2,\Z/p^2\Z)$.
\end{lemma}
\begin{proof}
If $g\in G\subseteq \GL(2,\Z/p^2\Z)$ is an element as in the statement of the Lemma, and $a\not\equiv d \bmod p$, it follows that $a$, $d$, and $a-d$ are units modulo $p^2$. If we conjugate $G$ by 
$$m:=\left(\begin{array}{cc}
1 & -bp/(a-d)\\
cp/(a-d) &1 
\end{array}\right)$$
if necessary (i.e.,~take $m^{-1}Gm$ instead of $G$), we may assume that $g=\left(\begin{array}{cc} a & 0 \\ 0 & d\end{array}\right)$ since, indeed, $m^{-1}gm$ is diagonal modulo $p^2$ of the indicated form). Moreover, we note that the conjugation by $m$ sends $G$ to another group that is congruent to a diagonal subgroup when reducing modulo $p$, since $m$ is congruent to the identity modulo $p$. 

Now, since $G$ is abelian, an arbitrary element $t=\left(\begin{array}{cc} a' & b'p \\ c'p & d'\end{array}\right)$ of $G$ must commute with $g$. And 
$g\cdot t \equiv t\cdot g \bmod p^2$
holds if and only if 
\begin{align*}
b'd & \equiv ab' \bmod p,\\
ac' & \equiv dc' \bmod p.
\end{align*}
Since $a$ and $d$ are units and $a-d\not\equiv 0 \bmod p$, this implies that $b'\equiv c'\equiv 0 \bmod p$, and therefore $t\in G$ must be diagonal modulo $p^2$.
\end{proof}

\begin{prop}\label{prop-borel2}
Let $E/\Q$ be an elliptic curve.  Let $p>2$ be a prime, let $G$ be the image of $\rho_{E,p}$, and assume that $G$ is contained in a Borel subgroup.  If $G$ is abelian, then there is a basis of $E[p]$ such that $G$ is diagonalizable (i.e.,~$G$ is contained in a split Cartan subgroup).
\end{prop}
\begin{proof}
Suppose that $G$ is a Borel subgroup. Since the determinant of $\rho_{E,p}$ is the cyclotomic character $\chi_p\colon \GQ \to (\Z/p\Z)^\times$, which is surjective, there must be a matrix in $G$ whose diagonal entries are distinct mod $p$ (otherwise all determinants would be squares in $(\Z/p\Z)^\times$). Then, Lemma \ref{lem-borel} implies that, without loss of generality, we can assume $G=G_dG_1$ in the notation of Definition \ref{defn-borel1}, and so $[G,G]=G_1$. Since $G$ is abelian, we must have that $G_1$ is trivial, and therefore $G=G_d$ is diagonal within $\GL(2,\Z/p\Z)$.
\end{proof}

\begin{cor}\label{cor-split}
Let $E/\Q$ be an elliptic curve, let $p>2$ be a prime, and let $G$ be the image of $\rho_{E,p^2}$. Suppose that $G$ is abelian, and further assume that $G\mod p$ is contained in a  Borel subgroup of $\GL(2,\Z/p\Z)$. Then, there is a basis of $E[p^2]$ such that $G$ is contained in a split Cartan subgroup of $\GL(2,\Z/p^2\Z)$.
\end{cor}
\begin{proof}
If $G$ is abelian, then $G\bmod p$ is abelian as well. By Proposition \ref{prop-borel2}, there is a basis of $E[p^2]$ such that $G\bmod p$ is diagonal. Moreover, as pointed out in the proof of the proposition, there is a matrix in $G$ whose diagonal entries are distinct modulo $p$. Hence, Lemma \ref{lem-borel3} applies, and there is a basis of $E[p^2]$ such that $G$ is diagonal in $\GL(2,\Z/p^2\Z)$, as claimed.
\end{proof}

\begin{prop}\label{prop-split}
Let $E/\Q$ be an elliptic curve.  Let $p>2$ be a prime, let $G$ be the image of $\rho_{E,p}$, and suppose that $G$ is abelian, and contained in the normalizer of a split Cartan subgroup of $\GL(2,\Z/p\Z)$. Then, $G$ must be diagonalizable.
\end{prop}
\begin{proof}
In order to abbreviate matrix notation, we define diagonal and anti-diagonal matrices:
$$D(a,b) = \left(\begin{array}{cc} a & 0 \\ 0 & b\end{array}\right),\quad A(c,d) = \left(\begin{array}{cc} 0 & c \\ d & 0\end{array}\right),$$
for any $a,b,c,d\in \F_p^\times$. With this notation, the split Cartan subgroup of $\GL(2,\F_p)$ is the subgroup $\spl = \{D(a,b) : a,b \in \F_p^\times\}$. It is easy to show that the normalizer of the split Cartan subgroup of $\GL(2,\F_p)$ is the subgroup
$$\mathcal{C}_{\text{sp}}^+ = \{D(a,b),\ A(c,d): a,b,c,d\in\F_p^\times\}.$$
Suppose $G\subseteq \mathcal{C}_{\text{sp}}^+$ is abelian, and $A(c,d)\in G$ for some $c,d\in\F_p^\times$. Then, if $D(a,b)\in G$, we have
$$A(bc,ad)=A(c,d)D(a,b) = D(a,b)A(c,d)= A(ac,bd),$$
and, therefore, $ac\equiv bc$, $ad\equiv bd\bmod p$, and so $a\equiv b \bmod p$. Hence, if $D(a,b)\in G$, then $a\equiv b\bmod p$. Similarly, if $A(c,d),A(e,f)\in G$, we have
$$D(cf,de)=A(c,d)A(e,f) = A(e,f)A(c,d)= D(de,cf),$$
and so $cf\equiv de\bmod p$ or, equivalently, $c/d\equiv e/f \bmod p$, i.e.,~there is $\varepsilon\in \F_p^\times$ such that if $A(c,d)\in G$, then $c\equiv \varepsilon d\bmod p$. In summary, so far we have shown that
$$G\subseteq \{D(a,a),A(\varepsilon c,c): a,c\in\F_p^\times \}\subseteq \mathcal{C}_{\text{sp}}^+$$
where $\varepsilon$ is fixed. Moreover, $G$ must contain an element $\sigma$ that corresponds to the image of a complex conjugation in $\GQ$. Thus, $\det(\rho_{E,p}(\sigma))=\chi_p(\sigma)\equiv -1 \bmod p$, and $(\rho_{E,p}(\sigma))^2$ must be the identity, thus, the eigenvalues of $\rho_{E,p}(\sigma)$ are $1$ and $-1$. Since the eigenvalues of $D(a,a)$ are $a,a$, and the eigenvalues of $A(\varepsilon c,c)$ are $\pm c\sqrt{\varepsilon}$, we must have $\sigma = A(\varepsilon c,c)$ where  $\varepsilon\equiv \mu^2\in(\F_p^\times)^2$ is a square modulo $p$, and $c\equiv \pm \mu^{-1}\bmod p$. Hence,
$$G\subseteq \{D(a,a),A(\mu^2 c,c): a,c\in\F_p^\times \}\subseteq \mathcal{C}_{\text{sp}}^+,$$
for some fixed $\mu\in\F_p^\times$. Since the matrices $D(a,a)$ and $A(\mu^2c,c)$, for any $a,c\in\F_p^\times$, fix the subspaces $\langle(\mu,1)\rangle$ and $\langle(-\mu,1)\rangle$, and $p>2$ (thus, the two subspaces are independent), it follows that $G$ is diagonalizable, as desired.
\end{proof}

\begin{prop}\label{prop-nonsplit}
Let $E/\Q$ be an elliptic curve.  Let $p>2$ be a prime, let $G$ be the image of $\rho_{E,p}$, and suppose that $G$ is abelian, and contained in the normalizer of a non-split Cartan subgroup of $\GL(2,\Z/p\Z)$. Then, $G$ must be diagonalizable.
\end{prop}
\begin{proof}
Suppose first that $G$ is contained in a non-split Cartan subgroup of $\GL(2,\Z/p\Z)$, i.e., 
$$G\subseteq \left\{ \left(\begin{array}{cc}
a & \varepsilon b \\
b & a
\end{array} \right) :a,b\in \Z/p\Z,\ (a,b)\not\equiv (0,0) \bmod p \right\}=C_{ns},$$
where $\varepsilon$ is a fixed quadratic non-residue modulo $p$. Let $\sigma\in \GQ$ be a complex conjugation. Then, $\det(\rho_{E,p}(\sigma))=\chi_p(\sigma)\equiv -1 \bmod p$, where $\chi_p$ is the $p$-th cyclotomic character. Moreover, $(\rho_{E,p}(\sigma))^2$ must be the identity. Since $C_{ns}$ is cyclic of order $p^2-1$, there is a unique element of order $2$, namely $-\operatorname{Id}$, but $\det(-\operatorname{Id})\equiv 1 \bmod p$, and we have reached a contradiction (because $p>2$ and $1\not\equiv -1 \bmod p$). Thus, $G\subset C_{ns}$ is impossible.

Now suppose that $G$ is abelian and contained in $C_{ns}^+$, the normalizer of $C_{ns}$ in $\GL(2,\Z/p\Z)$, i.e., 
$$G\subseteq \left\{ \left(\begin{array}{cc}
a & \varepsilon b \\
b & a
\end{array} \right), \left(\begin{array}{cc}
a & \varepsilon b \\
-b & -a
\end{array} \right) :a,b\in \Z/p\Z,\ (a,b)\not\equiv (0,0) \bmod p \right\}=C^+_{ns},$$
but $G\not\subset C_{ns}$. In order to abbreviate matrix notation, let $M(a,b)$ be an element of $C_{ns}$, and let $N(c,d)$ be an element of $C_{ns}^+$ not contained in $C_{ns}$. 

Since $G\not\subset C_{ns}$, there are $c,d$ such that $N(c,d)\in G$, with $d\not\equiv 0 \bmod p$. To see this, note that if there is $M(a,b)\in G$ with $b\not\equiv 0\bmod p$, then $M(a,b)\cdot N(c,0) = N(ac,-bc)$ with $bc\not\equiv 0 \bmod p$. Otherwise, if $G$ is generated by matrices of the form $M(a,0)$ and $N(c,0)$, then $G$ is contained in the split Cartan $C_s$, and therefore diagonalizable.

Further, if $M(a,b)\in G$, and $N(c,d)\in G$ with $d\not\equiv 0\bmod p$ is as above, since $G$ is abelian we have 
$$M(a,b)\cdot N(c,d)\equiv N(c,d)\cdot M(a,b)\bmod p,$$
which implies $2\varepsilon bd\equiv 0 \bmod p$. Since $p>2$ and $d\not\equiv 0\bmod p$, we conclude that $b\equiv 0\bmod p$. Hence,
$$G\subseteq \{M(a,0),N(e,f):a,e,f\in\Z/p\Z\}.$$
Now, if $G$ contains matrices $N(c,d)$ and $N(c',d')$, they also commute, and in particular
\begin{align*} 
N(c,d)\cdot N(c',d')&\equiv \left(\begin{array}{cc}
cc'-\varepsilon dd' & \varepsilon (cd'-c'd) \\
cd'-c'd & cc'-\varepsilon dd'
\end{array} \right)\\
&\equiv \left(\begin{array}{cc}
cc'-\varepsilon dd' & \varepsilon (c'd-cd') \\
c'd-cd' & cc'-\varepsilon dd'
\end{array} \right) \equiv N(c',d')\cdot N(c,d)\bmod p,\end{align*}
and therefore $2cd'\equiv 2c'd\bmod p$, and so $c'/c\equiv d'/d\equiv a \bmod p$, and $N(c',d')\equiv M(a,0)\cdot N(c,d)$. This shows that $G$ is a subgroup of the group generated by 
$$H=\{ N(c,d), M(a,0) : a\in (\Z/p\Z)^\times \}.$$
Since $G$ is the image of $\rho_{E,p}$, it must have full determinant, and therefore we can assume that $N(c,d)$ is of determinant $g$, a primitive root mod $p$, because the matrices $M(a,0)$ have square determinants (and the determinant of $N(ac,ad)\equiv M(a,0)\cdot N(c,d)$ is $a^2\cdot g$). 

Moreover, if $\sigma\in \GQ$ is a complex conjugation, we have  $\det(\rho_{E,p}(\sigma))=\chi_p(\sigma)\equiv -1 \bmod p$ and the trace of $\rho_{E,p}(\sigma)$ is zero. If $p\equiv 1\bmod 4$, then the only matrices with determinant $-1$ (a square mod $p$) are of the form $M(a,0)$, with trace $2a\not\equiv 0\bmod p$, so we would reach a contradiction, as there needs to be an element of $G$ that corresponds to the image of $\sigma$. 

If $p\equiv 3 \bmod 4$, then the eigenvalues of $N(c,d)$ are $\pm \lambda$ with $\lambda^2 \equiv c^2-\varepsilon d^2\equiv - \det(N(c,d)) \equiv -g \bmod p$, which is a quadratic residue because both $-1$ and $g$ are quadratic non-residues. Thus, $c^2-\varepsilon d^2$ is a quadratic residue, and $\pm \lambda \in \Z/p\Z$. Hence $N(c,d)$ has two distinct eigenvectors $u$ and $v$ that form a basis of eigenvectors, with eigenvalues $\lambda$ and $-\lambda$ respectively.  Since $G$ is contained in $H$, and $H$ is generated by diagonal matrices $M(a,0)$ (which fix the directions $u$ and $v$) and $N(c,d)$, then $H$ is a diagonalizable subgroup (with respect to the basis $\{u,v\}$), and so is $G$, as desired.
\end{proof}

\begin{thm}\label{thm-diagonal}
Let $E/\Q$ be an elliptic curve.  Let $G$ be the image of $\rho_{E,p}$.  If $p>2$ and $G$ is abelian, then there is a basis of $E[p]$ such that $G$ is diagonalizable.
\end{thm}
\begin{proof}
By the classification of maximal subgroups of $\GL(2,\Z/p\Z)$ in Theorem \ref{thm-serre2}, the group $G$ is a conjugate of a subgroup in one of the options (1)-(4). If $G$ is abelian, then it is clear that option (3) is impossible, as $G$ would have a non-abelian quotient. Moreover, if $G$ is abelian, then options (1), (2), and (4) imply that $G$ is diagonalizable, by Propositions \ref{prop-split}, \ref{prop-nonsplit}, and \ref{prop-borel2},  respectively. 
\end{proof}

\begin{cor}\label{cor-p5}
Let $E/\Q$ be an elliptic curve, let $p>2$ be a prime, and suppose that $\Q(E[p])/\Q$ is abelian. Then, the $\Q$-isogeny class of $E$ contains at least three distinct $\Q$-isomorphism classes, and $C_p(E)\geq 3$ (in the notation of Theorem \ref{thm-kenku}). In particular, $p\leq 5$.
\end{cor}
\begin{proof}
If $p>2$ and $\Q(E[p])/\Q$ is abelian, then Theorem \ref{thm-diagonal} says that there is a $\Z/p\Z$-basis $\{P,Q\}$ of $E[p]$, such that the image of $\rho_{E,p}$ is diagonalizable. Hence, $\langle P \rangle$ and $\langle Q\rangle$ are distinct Galois-stable cyclic subgroups of $E$. Therefore, $C_p(E)\geq 3$. By Theorem \ref{thm-kenku}, this is only possible if $p\leq 5$.
\end{proof}

\begin{prop}\label{prop-p2}
Let $E/\Q$ be an elliptic curve, let $p$ be a prime, and suppose that $\Q(E[p^2])/\Q$ is abelian. Then, $p=2$.
\end{prop}
\begin{proof}
Suppose that $\Q(E[p^{2}])/\Q$ is abelian, and $p>2$. In particular, $\Q(E[p])$ is abelian and Corollary \ref{cor-p5} says that $p\leq 5$. It follows that the image of $\rho_{E,p^2}$, call it $G$, is abelian, and Theorem \ref{thm-diagonal} implies that $G\bmod p$ is diagonal in some basis. Hence, Corollary \ref{cor-split} says that $G$ itself is also diagonalizable. Thus, there is a $\Z/p^2\Z$-basis $\{P,Q\}$ of $E[p^2]$ such that $\langle P\rangle$ and $\langle Q \rangle$ are $\GQ$-invariant. But then, $C_p(E)\geq 5$, as $\{\mathcal{O}\}$, $\langle P\rangle$, $\langle pP\rangle$, $\langle Q\rangle$, $\langle pQ\rangle$, are all invariant subgroups of order a power of $p$. This is impossible for $p=3$ or $5$ because $C_3(E)\leq 4$ and $C_5(E)\leq 3$, respectively. Thus, $p=2$.
\end{proof}

\begin{lemma}\label{lem-twists}
Let $E/\Q$ be an elliptic curve, and let $E'/\Q$ be a curve with $j(E)=j(E')$. Let $n> 2$, and put $G=\Gal(\Q(E[n])/\Q)$, $G'=\Gal(\Q(E'[n])/\Q)$. Then, 
\begin{enumerate}
\item If $j\neq 0,1728$, the curve $E'$ is a quadratic twist of $E$.
\item  If $E'$ is a quadratic twist of $E$, then either $G\cong G'$, or $G\cong G'\times \Z/2\Z$, or $G\times \Z/2\Z\cong G'$. If $n=2$, then $G\cong G'$. In particular, $G$ is abelian if and only if $G'$ is abelian.
\item If $G=\Gal(\Q(E[n])/\Q)$ is a Borel (resp.~split Cartan) subgroup of $\GL(2,\Z/n\Z)$, then $G'$ is also a Borel (resp.~split Cartan) subgroup. 
\end{enumerate}
\end{lemma}
\begin{proof}
The first part can be found, for instance, in Lemma 9.6 of \cite{lozano0}. The rest is all immediate from the equation $\phi\circ \rho_{E',m}\circ \phi^{-1}=\rho_{E,m}^{\xi^{-1}}=[\xi^{-1}]\circ \rho_{E,m}$ that relates the representation on $E'[m]$ with that of $E[m]$, where $\phi \colon E'\to E$ is the $\Q(\sqrt{D})$-isomorphism between $E'$ and $E$, and $\xi(\sigma)=\sqrt{D}/\sigma(\sqrt{D})$, where $E'$ is the quadratic twist of $E$ by $D\in\Q^\times$.
\end{proof}

\begin{lemma}\label{lem-6abelian}
If $\Q(E[6])/\Q$ is abelian, then $\Gal(\Q(E[2])/\Q)\cong \Z/2\Z$.
\end{lemma}
\begin{proof}
Suppose $\Gal(\Q(E[6])/\Q)$ is abelian. We will show in Section \ref{sec-cm} that this is impossible for CM curves, so we may assume in particular that $j\neq 0,1728$. Then, $\Gal(\Q(E[2])/\Q)$ is itself abelian, and therefore either trivial, or isomorphic to $\Z/2\Z$ or $\Z/3\Z$. If $\Gal(\Q(E[2])/\Q)$ is trivial, then $\Q(E[2])=\Q$ and all $2$-torsion points are rational. In particular, $C_2(E)\geq 4$ (with notation as in Theorem \ref{thm-kenku}). If in addition $\Q(E[3])$ is abelian, then Theorem \ref{thm-diagonal} implies that $C_3(E)\geq 3$. Thus, we would have $C(E)\geq C_2(E)\cdot C_3(E)\geq 12>8$, which contradicts Theorem \ref{thm-kenku}. Hence, this is impossible.

It remains to eliminate the possibility of $\Gal(\Q(E[2])/\Q)\cong \Z/3\Z$ and $\Q(E[3])$ abelian. Let us assume this occurs for a contradiction. Since $\Q(E[3])$ is abelian over $\Q$, Theorem \ref{thm-diagonal} says that $j(E)$ appears as a non-cuspidal point in $X_s(3)$, the modular curve associated to a split Cartan image for $p=3$.  The curve $X_s(3)$ is of genus $0$ (see \cite{zywina1}) and the $j$-line $j \colon X_s(3)\to \PP^1(\Q)$ is equal to the one for $X(3)$ (see Section \ref{sec-Q3} below). Since $j\neq 0,1728$, Lemma \ref{lem-twists} implies that there exists some $t_0\in\Q$ such that $E$ is a quadratic twist of the elliptic curve $E'$ with Weierstrass model:
$$
E:y^2=x^3-27t_0(t_0^3+8)x+54(t_0^6-20t_0^3-8).
$$
Moreover, $E'$ also satisfies that $\Gal(\Q(E'[3])/\Q)$ is abelian, and $\Gal(\Q(E'[2])/\Q)\cong \Z/3\Z$. In particular, we must have that the discriminant of $E'/\Q$ is a perfect square. Since $\Delta_{E'}=2^{12}3^{9}(t_0^3-1)^3$, this means that there is $v_0\in \Q$ such that $v_0^2=3(t_0^3-1)$. The curve $C:v^2=3(u^3-1)$ is $\Q$-isomorphic to $E'':y^2=x^3-27$ by the change of variables $(x,y)=3(u,v)$ and $E''$ is the elliptic curve  \texttt{36a3} that satisfies $E''(\Q)=\{\mathcal{O},(3,0)\}$. Thus, the only affine point on $C$ is $(1,0)$, and therefore, $E'$ cannot exist and neither can $E/\Q$ with the given properties. This shows that the only possibility when $\Q(E[3])/\Q$ is abelian is  $\Gal(\Q(E[2])/\Q)\cong \Z/2\Z$.
\end{proof}

\begin{lemma}\label{lem-10abelian}
The extension $\Q(E[10])/\Q$ is never abelian for an elliptic curve $E$ defined over the rationals.
\end{lemma}
\begin{proof}
We will show in Section \ref{sec-cm} that this is impossible for CM curves, so we may assume in particular that $j\neq 0,1728$. Let us suppose for a contradiction that $\Q(E[10])/\Q$ is abelian for an elliptic curve $E/\Q$. Then, $\Q(E[2])/\Q$ and $\Q(E[5])/\Q$ are abelian. In particular, $\Gal(\Q(E[2])/\Q)$ is trivial, or isomorphic to $\Z/2\Z$, or $\Z/3\Z$. In the first two cases, $E$ has a non-trivial rational $2$-torsion point $P\in E(\Q)[2]$, and therefore a $2$-isogeny $E\to E/\langle P \rangle$. Since $ \Q(E[5])/\Q$ is abelian, Corollary \ref{cor-p5} implies that $E$ contains two distinct non-trivial, Galois-stable, cyclic subgroups of order $5$, say $\langle Q_1\rangle$ and $\langle Q_2\rangle$. Then,
$$E/\langle P+Q_1\rangle \stackrel{2}{\longrightarrow} E/\langle Q_1\rangle \stackrel{5}{\longrightarrow} E \stackrel{5}{\longrightarrow} E/\langle Q_2 \rangle$$
is an isogeny of degree $50$, which cannot exist by Theorem \ref{thm-ratnoncusps}. Hence, $\Gal(\Q(E[2])/\Q)$ cannot be trivial or $\Z/2\Z$. 

Suppose then $\Gal(\Q(E[2])/\Q)\cong \Z/3\Z$ for some elliptic curve $y^2=f(x)$, where $f(x)$ is an irreducible cubic polynomial with distinc roots (over $\CC$). Then, $\Gal(\Q(E[2])/\Q)=\Gal(f)=\Z/3\Z$ and this occurs if and only if $\operatorname{Disc}(f)$ is a perfect square. Since $\Q(E[5])/\Q$ is abelian, it follows from Theorem \ref{thm-diagonal} that $j(E)$ appears as a non-cuspidal point in $X_s(5)$, i.e.,~the modular curve associated to the split-Cartan image for $p=5$. The curve $X_s(5)$ is of genus $0$ (see \cite{zywina1}; see also Section \ref{sec-Q5} below), and the $j$-line $j\colon X_s(5)\to \PP^1(\Q)$ can be given as follows
$$j_{X_s(5)}(t)=j(t):=\frac{(t^2+5t+5)^3(t^4+5t^2+25)^3(t^4+5t^3+20t^2+25t+25)^3}{(t(t^4+5t^3+15t^2+25t+25))^5}.$$
Thus, $j(E)=j(t_0)$ for some $t_0\in\Q^*$. It follows that $E$ is a quadratic twist of the curve
$$E': y^2 + xy = x^3 - \frac{36}{(j(t_0) - 1728)}x - \frac{1}{(j(t_0) - 1728)}$$ 
with $j(E')=j(t_0)$.  Since $E$ and $E'$ are quadratic twists of each other, it follows from Lemma \ref{lem-twists} that $E'$ also satisfies $\Gal(\Q(E'[5])/\Q)$ is abelian, and $\Gal(\Q(E'[2])/\Q)\cong \Z/3\Z$. In particular, the discriminant of $E'$ must be a perfect square. In this case we have 
$$
\mbox{\small
$\Delta_{E'}=\frac{(t_0(t_0^4 + 5 t_0^3 + 15t_0^2 + 25t_0 + 25))^5}{(t_0^2 + 2t_0 + 5)^3}
\left(\frac{(t_0^2 + 5t_0 + 5)(t_0^4 + 5 t_0^2 + 25)(t_0^4 + 5t_0^3 + 20t_0^2 + 25t_0 + 25)}{(t_0^2 - 5)(t_0^4 + 15t_0^2 + 25)(t_0^4 + 4t_0^3 + 9t_0^2 + 10t_0 + 5)(t_0^4 + 10t_0^3 + 45t_0^2 + 100t_0 + 125)}\right)^6$}
$$
This implies that there is $y_0\in\Q$ such that
$$
y_0^2=t_0(t_0^2 + 2t_0 + 5)(t_0^4 + 5 t_0^3 + 15t_0^2 + 25t_0 + 25).
$$
In other words, $(t_0,y_0)$ is an affine point on the hyperelliptic genus $3$ curve:
$$
C\,:\,y^2=x(x^2 + 2x + 5)(x^4 + 5 x^3 + 15x^2 + 25x + 25).
$$
Note that $(1: 0:0),(0:0:1)\in C(\Q)$ and $C$ has the automorphism $\tau$ of degree $2$ given by 
$$
\tau\,:\, C \longrightarrow C,\qquad  \tau(x,y)=\left(\frac{5}{x},\frac{25y}{x^4}\right).
$$
Therefore, we have a degree $2$ morphism:
$$
\pi\,:\, C \longrightarrow E''=C/\langle\tau\rangle,\qquad \pi(x,y)=\left(\frac{x^2+2x+5}{x},-\frac{y}{x^2}\right),
$$
where $E''\,:\, v^2 = u^3 + u^2 - u$ and $\pi(\{(1: 0:0),(0: 0:1)\})=\mathcal O\in E''(\Q)$. Now, $E''$ is the elliptic curve \texttt{20a2} that satisfies $E''(\Q)\simeq \Z/6\Z$. In particular:
$$
E''(\Q)=\{(\pm 1,:\pm 1:1),(0:0:1)\}\cup \{\mathcal O\}.
$$
It is an straightforward computation to check that $\pi^{-1}(E''(\Q))\cap C(\Q)=\{(1: 0: 0),(0: 0:1)\}$. From here it follows that $C(\Q)=\{(1 : 0 : 0), (0 : 0 : 1)\}$. Hence, the only rational points on $C$ correspond to cusps of $X_s(5)$. We conclude that $E'$ cannot exist, and neither does $E$, i.e.,~$\Gal(\Q(E[10])/\Q)$ cannot be abelian for an elliptic curve $E/\Q$.
\end{proof}

\section{Abelian division fields in CM curves}\label{sec-cm}

In this section we study elliptic curves $E/\Q$ with complex multiplication by an order in an imaginary quadratic field $K=\Q(\sqrt{-d})$, and characterize, for each isomorphism class, when $\Q(E[n])$ is abelian, and when $\Q(E[n])=\Q(\zeta_n)$. Our results are summarized in Theorem \ref{thm-maincm} and Table 1 in Section \ref{table1}, below. 

\begin{lemma}\label{lem-j0} Let $E/\Q$ be an elliptic curve with $j=0$ and given by a model $y^2=x^3+s$ over $\Q$. Then, $E/\Q$ has a rational $3$-isogeny. In addition,
\begin{enumerate}
\item $E/\Q$ has a rational $2$-isogeny if and only if $s=t^3$, for some non-zero integer $t$, or equivalently, if $E/\Q$ is a quadratic twist of $y^2=x^3+1$.
\item $E/\Q$ has two distinct rational $3$-isogenies if and only if $s=2t^3$, for some non-zero integer $t$, or equivalently, if $E/\Q$ is a quadratic twist of $y^2=x^3+2$.
\end{enumerate}
\end{lemma}
\begin{proof}
This is immediate from the formulas for the $2$nd and $3$rd division polynomials, $\psi_2$ and $\psi_3$, for the surface $y^2=x^3+s$, which are given by
$$\psi_2(x)= x^3+s, \ \text{ and } \ \psi_3(x)=x(x^3+4s).$$
\end{proof}

\begin{lemma}\label{lem-j1728} Let $E/\Q$ be an elliptic curve with $j=1728$ and given by a model $y^2=x^3+sx$ over $\Q$. Then, $E/\Q$ has a rational $2$-isogeny. In addition,
\begin{enumerate}
\item $E/\Q$ has three rational $2$-isogenies if and only if $s=-t^2$, for some non-zero integer $t$, or equivalently, if $E/\Q$ is a quadratic twist of $y^2=x^3-x$.
\item $E/\Q$ has a rational $4$-isogeny if and only if $s=t^2$, for some non-zero integer $t$, or equivalently, if $E/\Q$ is a quadratic twist of $y^2=x^3+x$.
\end{enumerate}
\end{lemma}
\begin{proof}
This is immediate from the formulas for the $2$nd and $4$th division polynomials, $\psi_2$ and $\psi_4$, for the surface $y^2=x^3+sx$, which are given by
$$\psi_2(x)= x(x^2+s), \ \text{ and } \ \psi_4(x)=x(x^2-s)(x^2+s)(x^4+6sx^2+s^2),$$
so that $\psi_4(x)/\psi_2(x)=(x^2-s)(x^4+6sx^2+s^2)$.
\end{proof}

\begin{thm}\label{thm-maincm} Let $E/\Q$ be an elliptic curve with complex multiplication by an imaginary quadratic field $K$, with discriminant $d_K$. If there is an integer $n\geq 2$  such that $\Q(E[n])=\Q(\zeta_n)$, then $n=2$, or $3$. More generally, if $\Q(E[n])/\Q$ is abelian, then  $n=2,3$, or $4$.  Moreover, $\Gal(\Q(E[n])/\Q)$ is isomorphic to one of the following groups:
\begin{center}
\renewcommand{\arraystretch}{1.3}
\begin{tabular}{|c||c|c|c|c|}
\hline
$n$ & \multicolumn{2}{c|}{$2$}  & $3$ & $4$ \\
\hline
$d_K$ & $-4$ & $-3,-7,-8$  & $-3$ & $-4$ \\
\hline 
\multirow{2}{*}{$\Gal(\Q(E[n])/\Q)$}  & $\{0\}$ & $\Z/2\Z$  & $\Z/2\Z$ & $(\Z/2\Z)^2$  \\
  &  $\Z/2\Z$ &  & $(\Z/2\Z)^2$ & $(\Z/2\Z)^3$  \\
\hline
\end{tabular}
\end{center}
\end{thm} 
\begin{proof}
Since there are only $13$ complex multiplication $j$-invariants over $\Q$, it suffices to check the structure of the division fields for an elliptic curve over $\Q$ with each one of these $j$-invariants, and their twists. By Corollaries \ref{cor-p5}, and Proposition \ref{prop-p2}, it suffices to check the $2^n$, $3$, and $5$ division fields, and their possible compositums. Moreover, if $\Q(E[2^n])$ is abelian, then there is a $2$-isogeny; and if $\Q(E[p])$, with $p=3$, or $5$, is abelian, then $E/\Q$ has two independent $p$-isogenies. Since no $j$-invariant with CM has a $5$-isogeny, $\Q(E[5])$ is never abelian. 

The CM $j$-invariants with a $2$-isogeny are those with $d_K= -3,-4,-7$, or $-8$. For curves $E$ with one of these invariants, we have $\Q(E[2])=\Q(\sqrt{\pm d_K})$ or $\Gal(\Q(E[2])/\Q)\cong S_3$, and in particular, $\Q(E[2])=\Q(\zeta_2)=\Q$ may only occur for $d_K=-4$.

If $\Q(E[4])$ is abelian, then, in particular, $\Q(E[2])$ is abelian. We distinguish several cases depending on the $j$-invariant of the curve:
\begin{enumerate}
\item  If $j=0$, then $E:y^2=x^3+t^3$, by Lemma \ref{lem-j0}. However, in this case, Corollary \ref{cor-E4} implies 
$$\Q(E[4])=\Q\left(i,\sqrt{-3},\sqrt{t},\sqrt{6\sqrt{3}-9}\right)$$
which satisfies $\Gal(\Q(E[4])/\Q)\cong D_4$ or $D_4\times \Z/2\Z$, where $D_4$ is the dihedral group of order $8$. Hence the extension is not abelian.
\item If $j=1728$, then $E:y^2=x^3+sx$, for some $s\in \Q$. Corollary \ref{cor-E4} implies that 
$$\Q(E[4])=\Q\left(\sqrt[4]{-s},\sqrt{-\sqrt{-s}},\sqrt{2\sqrt{-s}}\right)=\Q\left(i,\sqrt{2},\sqrt[4]{-s}\right)=\Q\left(\zeta_8,\sqrt[4]{-s}\right).$$
Then, $G=\Gal(\Q(E[4])/\Q)$ can have $4$ isomorphism types according to the value of $s$:
\begin{enumerate}
\item If $s=\pm t^4$ or $\pm 4t^4$, for some $t\in \Q$, then $\Q(E[4])=\Q(\zeta_8)$, and $G\cong (\Z/2\Z)^2$.
\item If $s=\pm t^2$, for some $t\in \Q$ (but $t\neq 2\cdot \square$), then $\Q(E[4])=\Q(\zeta_8,\sqrt{t})$, and $G\cong (\Z/2\Z)^3$.
\item If $s=\pm 2t^2$, for some $t\in \Q$, then $G\cong D_4$, which is non-abelian.
\item Otherwise, $G\cong D_4\times \Z/2\Z$, which is also non-abelian of order $16$.
\end{enumerate}

\item If $j=2^4\cdot 3^3\cdot 5^3$, then $E$ is a quadratic twist of $E':y^2 = x^3 - 15x + 22$. However, one can verify that $\Gal(\Q(E'[4])/\Q)$ is non-abelian, and therefore $\Gal(\Q(E[4])/\Q)$ is non-abelian also by Lemma \ref{lem-twists}.  Similarly, one checks that the $4$-torsion is not abelian in the cases of $j=2^3\cdot 3^3\cdot 11^3$, and $-3^3\cdot 5^3$, and $3^3\cdot 5^3\cdot 17^3$, and $2^6\cdot 5^3$.

\end{enumerate}

If $\Q(E[8])/\Q$ is abelian, then $\Q(E[4])$ is abelian, and therefore, $E: y^2=x^3\pm t^2 x$, for some $t\in\Q$. Thus, $E$ is a quadratic twist of either $E_+:y^2=x^3+x$ or $E_-:y^2=x^3-x$. By Lemma \ref{lem-twists}, if $\Gal(\Q(E[8])/\Q)$ is abelian, then so is $\Gal(\Q(E_\pm[8])/\Q)$. However, the field $\Q(E_\pm[8])$ contains the number field generated by a root of 
\begin{align*}
f_+ &= x^4 - 4x^3 - 2x^2 - 4x + 1,\\
f_- &= x^8 + 20x^6 - 26x^4 + 20x^2 + 1,
\end{align*}
and $\Gal(f_\pm)\cong D_4$ (note: the polynomials $f_\pm$ are divisors of the $8$th division polynomial of $E_\pm$). Thus, $\Gal(\Q(E_\pm[8])/\Q)$ and $\Gal(\Q(E[8])/\Q)$ are non-abelian. In particular, we have shown $\Q(E[8])/\Q$ is never abelian for a CM curve.

The only CM $j$-invariant with a $3$-isogeny is $j=0$ and Lemma \ref{lem-j0} says that the only curves with $j=0$ and two distinct $3$-isogenies are those of  the form $E: y^2=x^3+16t^3$, for some $t\in \Q^\times$. In this case,  $\Q(E[3])=\Q(\sqrt{-3},\sqrt{t})$. Hence, $\Q(E[3])=\Q(\zeta_3)=\Q(\sqrt{-3})$ occurs for the $\Q$-isomorphism class of $E: y^2=x^3+16$.

Finally, it remains to show that $\Q(E[6])$ and $\Q(E[12])$ cannot be abelian for any CM curve. However, for those CM $j$-invariants with $\Q(E[3])/\Q$ abelian, we have $\Gal(\Q(E[2])/\Q)\cong S_3$, and therefore neither $\Q(E[6])$ nor $\Q(E[12])$ can be abelian over $\Q$. This concludes the proof of the theorem.
\end{proof}

\section{Proof of the main theorem}\label{sec-mainproof}

We begin the proof of our main Theorem \ref{thm-main}, by dealing with the case when $n$ is a power of $2$.
\begin{thm}\label{thm-2torsab}
Let $E/\Q$ be an elliptic curve without CM such that $\Q(E[2^n])/\Q$ is an abelian extension for some $n$. Then, $n\leq 3$. Moreover, $G_n=\Gal(\Q(E[2^n])/\Q)$ is isomorphic to one of the following groups:
\begin{align}
G_1\cong  \begin{cases}
\{0\} \text{, or}\\
\Z/2\Z \text{, or}\\
\Z/3\Z,
\end{cases},\ G_2\cong (\Z/2\Z)^k, \text{ with } 1\leq k\leq 4,\ \text{ or }\  G_3\cong (\Z/2\Z)^j, \text{ with } 4\leq j\leq 6.
\end{align}
In particular, if $\Q(E[4])/\Q$ or $\Q(E[8])/\Q$ is abelian, then it is polyquadratic (i.e.,~contained in $\Q(2^\infty)$, with notation as in Theorem \ref{thm-fujita}).
\end{thm}
\begin{proof}
Let $n\geq 1$ be fixed and let $E/\Q$ be an elliptic curve such that $\Q(E[2^n])/\Q$ is an abelian extension. Then, the image $\rho_{E,2^n}(\GQ)\subset \GL(2,\Z/2^n\Z)$ is abelian. In other words, the mod $2^n$ reduction of the $2$-adic image $\rho_{E,2^\infty}(\GQ)$ is an abelian subgroup of $\GL(2,\Z/2^n\Z)$. By Theorem \ref{thm-rzb}, there are precisely $1208$ possibilities for the $2$-adic image $\rho_{E,2^\infty}(\GQ)$ up to conjugacy, and these subgroups have been explicitly described in \cite{rouse}. More concretely, the authors of \cite{rouse} describe a set $\mathcal{H}$ of $1208$ pairs $(H,2^N)$ where  $H$ is a subgroup of $\GL(2,\Z/2^N\Z)$, such that $(H,2^N)\in \mathcal{H}$ if and only if there is an elliptic curve $E/\Q$ such that  $\rho_{E,2^\infty}(\GQ)$ is the full inverse image of $H$ under the reduction map modulo $2^N$. 

Thus, to compute all the possible cases when $\Q(E[2^n])/\Q$ is an abelian extension, it suffices to:
\begin{enumerate}
\item[(i)] Find all the pairs $(H,2^N)\in \mathcal{H}$ such that $H \bmod 2^n$ is abelian (within $\GL(2,\Z/2^n\Z)$), for $n\leq N$, and
\item[(ii)]  Find all the pairs $(H,2^N)\in \mathcal{H}$ such that the full inverse image of $H$ in $\GL(2,\Z/2^n\Z)$ under reduction modulo $2^N$ is abelian (as a subgroup of $\GL(2,\Z/2^n\Z)$), for $n> N$.
\end{enumerate}
We have used Magma \cite{magma} in order to find all the groups $H$ as above. We follow the labels as in \cite{rouse}, where they list groups by a number from $1$ to $727$, or by a label of the form $Xks$, where $1\leq k\leq 243$ and $s$ is a letter. Below, we list all the pairs $(\text{<label>},2^N)$ that fall in each category (i) and (ii) as above:
\begin{enumerate}
\item[(i)] (\texttt{X2},2), (\texttt{X6},2), (\texttt{X8},2), (\texttt{X24},4), (\texttt{X25},4), (\texttt{X27},4), (\texttt{X58},4), (\texttt{X60},4), (\texttt{X183},8), (\texttt{X187},8), (\texttt{X189},8), and\\
(\texttt{X8b}, 4), (\texttt{X8d}, 4), (\texttt{X24d}, 4), (\texttt{X24e}, 4), (\texttt{X25h}, 4), (\texttt{X25i}, 4), (\texttt{X25n}, 4), (\texttt{X27f}, 4), (\texttt{X27h}, 4), (\texttt{X58a}, 8), (\texttt{X58b}, 8), (\texttt{X58c}, 8), (\texttt{X58d}, 8), (\texttt{X58e}, 8), (\texttt{X58f}, 8), (\texttt{X58g}, 8), (\texttt{X58h}, 8), (\texttt{X58i}, 4), (\texttt{X58j}, 8), (\texttt{X58k}, 8), (\texttt{X60d}, 4), (\texttt{X183a}, 8), (\texttt{X183b}, 8), (\texttt{X183c}, 8), (\texttt{X183d}, 8), (\texttt{X183e}, 8), (\texttt{X183f}, 8), (\texttt{X183g}, 8), (\texttt{X183h}, 8), (\texttt{X183i}, 8), (\texttt{X183j}, 8), (\texttt{X183k}, 8), (\texttt{X183l}, 8), (\texttt{X187a}, 8), (\texttt{X187b}, 8), (\texttt{X187c}, 8), (\texttt{X187d}, 8), (\texttt{X187e}, 8), (\texttt{X187f}, 8), (\texttt{X187g}, 8), (\texttt{X187h}, 8), (\texttt{X187i}, 8), (\texttt{X187j}, 8), (\texttt{X187k}, 8), (\texttt{X187l}, 8), (\texttt{X189a}, 8), (\texttt{X189b}, 8), (\texttt{X189c}, 8), (\texttt{X189d}, 8), (\texttt{X189e}, 8), (\texttt{X189h}, 8), (\texttt{X189i}, 8), (\texttt{X189j}, 8), (\texttt{X189k}, 8), (\texttt{X189l}, 8).
\item[(ii)] (\texttt{X8},4), (\texttt{X58},8), and (\texttt{X58i}, 8).
\end{enumerate}
Hence, $8$ is the maximum level $2^n$ such that $\Q(E[2^n])/\Q$ is abelian. Moreover, for each $H$ as above, we have calculated (again using Magma) the Galois group $\Gal(\Q(E[2^n])/\Q)$. The results are summarized in Tables 3 and 4 in Sections \ref{table3} and \ref{table4} respectively (we have also included the torsion structure defined over $\Q$ for each type of abelian image) and we conclude the statement of the theorem.
\end{proof}

We are finally ready to put together the proof of the main theorem.

\begin{thm}\label{thm-main1}
Let $E/\Q$ be an elliptic curve, let $n\geq 2$, and suppose that $\Q(E[n])/\Q$ is abelian. Then, $n=2,3,4,5,6$, or $8$. Moreover, $\Gal(\Q(E[n])/\Q)$ is isomorphic to one of the following groups:
\begin{center}
\renewcommand{\arraystretch}{1.2}
\begin{tabular}{|c||c|c|c|c|c|c|}
\hline
$n$ & $2$ & $3$ & $4$ & $5$ & $6$ & $8$\\
\hline 
\multirow{4}{*}{$\Gal(\Q(E[n])/\Q)$}  & $\{0\}$ & $\Z/2\Z$ & $\Z/2\Z$ & $\Z/4\Z$ & $(\Z/2\Z)^2$ & $(\Z/2\Z)^4$ \\
 & $\Z/2\Z$ & $(\Z/2\Z)^2$ & $(\Z/2\Z)^2$ & $\Z/2\Z\times \Z/4\Z$ & $(\Z/2\Z)^3$ & $(\Z/2\Z)^5$ \\
 & $\Z/3\Z$ &  & $(\Z/2\Z)^3$ & $(\Z/4\Z)^2$ & & $(\Z/2\Z)^6$\\
 &  &  & $(\Z/2\Z)^4$ & & & \\
 \hline
\end{tabular}
\end{center} 
\end{thm}
\begin{proof}
The elliptic curves with CM have already been treated in Theorem \ref{thm-maincm}, so we assume here that $E/\Q$ has no CM.  Let $n=2^{e_1}\cdot p_2^{e_2}\cdots p_r^{e_r}$ be a prime factorization of $n$, where each $p_i$ is a distinct prime (here $p_1=2$) with $e_i\geq 1$ for $i\geq 2$, and $e_1\geq 0$. Since $\Q(E[p_i^{e_i}])\subseteq \Q(E[n])$, it follows that $\Q(E[2^{e_1}])/\Q$ is abelian, and therefore $e_1\leq 3$ by Theorem \ref{thm-2torsab}. Similarly, $\Q(E[p_i])/\Q$ is abelian. Hence, Corollary \ref{cor-p5} says that $p_i\leq 5$ and $C_{p_i}(E)\geq 3$ if $p_i>2$. Thus, if both $3$ and $5$ appeared in the factorization of $n$, we would have $C(E)\geq C_3(E)\cdot C_5(E)\geq 9$ which is impossible by Theorem \ref{thm-kenku}. Hence, $n=2^{e_1}p^{e_2}$ with $p=3$ or $5$, and $e_1\leq 3$.

Suppose $e_2\geq 2$ and $p=3$ or $5$. Then, $\Q(E[p^{e_2}])/\Q$ is abelian, but Proposition \ref{prop-p2} says that this is impossible. Hence, we have reached a contradiction and $e_2\leq 1$, i.e.,~$n$ is a divisor of $24$ or $40$. However, if $e_1\geq 1$ and $e_2=1$ with $p=5$, then $\Q(E[10])/\Q$ would be abelian which is impossible by Lemma \ref{lem-10abelian}. Thus, $n$ is a divisor of $5$ or $24$.

Since $\Q(E[12])\subseteq \Q(E[24])$, it suffices to show that $\Q(E[12])$ cannot be abelian. If $\Q(E[12])$ was abelian, then $\Q(E[6])$ is abelian. Lemma \ref{lem-6abelian} says that we must have $\Gal(\Q(E[2])/\Q)\cong \Z/2\Z$ and so $E/\Q$ must have exactly one non-trivial $2$-torsion point defined over $\Q$. Since $\Q(E[4])/\Q$ is also abelian, then the image of $G=\rho_{E,4}(\GQ)$ must be one of the groups given in Table 3 (\S \ref{table3}) such that $E(\Q)[2]\cong \Z/2\Z$. Thus, $G$ is one of
$$
\texttt{X27},\ \texttt{X60},\ \texttt{X27f},\ \texttt{X27h},\  \text{ or }\ \texttt{X60d},$$
in the notation of the subgroups described in \cite{rouse}. However, each of these groups have an invariant cyclic subgroup of order $4$ or, in other words, the associated elliptic curves $E$ have a $4$-isogeny. If so, then $C_2(E)\geq 3$ (in the notation of Theorem \ref{thm-kenku}), and $C_3(E)\geq 3$ because of Theorem \ref{thm-diagonal}. Therefore, $C(E)\geq C_2(E)\cdot C_3(E)=9>8$, which is impossible by Theorem \ref{thm-kenku}. It follows that neither $\Q(E[12])/\Q$ nor $\Q(E[24])/\Q$ can be abelian. This shows that $n=2,3,4,5,6$, or $8$.

The possible structures for $\Gal(\Q(E[n])/\Q)$ when $n$ is a power of $2$ were shown in Theorem \ref{thm-2torsab}. If $G=\Gal(\Q(E[3])/\Q)$ is abelian, then Theorem \ref{thm-diagonal} says that $G$ is diagonalizable in $\GL(2,\Z/3\Z)$. Thus, $G\cong \Z/2\Z$ or $(\Z/2\Z)^2$. If $G=\Gal(\Q(E[5])/\Q)$ is abelian, then Theorem \ref{thm-diagonal} says that $G$ is diagonalizable in $\GL(2,\Z/5\Z)$. Hence, $G\subseteq (\Z/5\Z)^\times \times (\Z/5\Z)^\times\cong (\Z/4\Z)^2$. Since $\Q(\zeta_5)\subseteq \Q(E[5])$, the only possibilities are $G\cong \Z/4\Z$, or $\Z/2\Z\times \Z/4\Z$, or $(\Z/4\Z)^2$. If $G=\Gal(\Q(E[6])/\Q)$ is abelian, then Lemma \ref{lem-6abelian} says that $\Gal(\Q(E[2])/\Q)\cong \Z/2\Z$ (and, therefore, there is a $2$-torsion point defined over $\Q$). If $\Q(E[3])/\Q$ is abelian, then we know that $\Gal(\Q(E[3])/\Q)\cong (\Z/2\Z)^k$ for $k=1$ or $2$. Hence, $G\cong (\Z/2\Z)^2$ or $(\Z/2\Z)^3$. We will show in Section \ref{sec-param} that all these structures do indeed occur. This concludes the proof of the theorem.
\end{proof}

\begin{thm}\label{thm-main2}
Let $E/\Q$ be an elliptic curve without complex multiplication, let $n\geq 2$, and suppose that $\Q(E[n])=\Q(\zeta_n)$, where $\zeta_n$ is a primitive $n$-th root of unity. Then, $n=2,3,4$, or $5$.
\end{thm}
\begin{proof}
By Theorem \ref{thm-main1} we know that $n=2,3,4,5,6$, or $8$. Hence, it suffices to show that $\Q(E[n])=\Q(\zeta_n)$ is not possible when $n=6$ and $n=8$. 

Suppose first that $\Q(E[6])=\Q(\zeta_{6})=\Q(\zeta_3)=\Q(\sqrt{-3})$. Then, $E/\Q$ is a curve defined over $\Q$ such that $E(\Q(\sqrt{-3}))_\text{tors}$ contains a subgroup of the form $\Z/6\Z\times \Z/6\Z$. But this is impossible by Theorem \ref{thm-najman2}.

Now, if $\Q(E[8])=\Q(\zeta_8)$, then $\Gal(\Q(E[8])/\Q)=\Gal(\Q(i,\sqrt{2})/\Q)\cong (\Z/2\Z)^2$, but we have seen in Theorem \ref{thm-main1} that $\Gal(\Q(E[8])/\Q)\cong (\Z/2\Z)^k$ with $k=4,5,$ or $6$. Hence, $\Q(E[8])=\Q(\zeta_8)$ is impossible, and this concludes the proof of the theorem.
\end{proof}
\section{Parametrizations}\label{sec-param}

In this section we describe the parametrizations of elliptic curves $E$ over $\Q$ with cyclotomic or abelian division fields $\Q(E[n])$. 

\subsection{$\Q(E[2])$}\label{sec-Q2}

If we write $E: y^2=f(x)=x^3+ax+b$ with discriminant $\Delta_E\neq 0$, then we have $\Gal(\Q(E[2])/\Q)=\Gal(f)$, and there are three possibilities for $G=\Gal(\Q(E[2])/\Q)$ abelian: $G$ is trivial, $\Z/2\Z$, or $\Z/3\Z$.
\begin{enumerate}
\item $G$ is trivial if and only if $f(x)$ is not irreducible over $\Q$ and $\Delta_E$ is a perfect square. Equivalently, these are the elliptic curves of the form
$$E_{c,d}: y^2=x(x-c)(x-d),$$
with non-zero and distinct rational numbers $c,d$.
\item $G\cong \Z/2\Z$ if and only if $f(x)$ is not irreducible over $\Q$ and $\Delta_E$ is not a perfect square.  Equivalently, these are the elliptic curves of the form
$$E_{e,f}: y^2=x(x^2+ex+f), \text{ with } e^2-4f\neq \square.$$
\item $G\cong \Z/3\Z$ if and only if $f(x)$ is irreducible over $\Q$ and $\Delta_E$ is a perfect square.  Equivalently, these are the elliptic curves of the form
$$E_{a,b}: y^2=x^3+ax+b, \text{ with } -(4a^3+27b^2) = \square.$$
For instance, a one parameter family of such curves is given by the so-called simplest cubic fields \cite{shanks}:
$$E_t: y^2=x^3-tx^2-(t+3)x-1,$$
with $\Delta_t = (t^2+3t+9)^2$.
\end{enumerate}

\subsection{$\Q(E[3])$}\label{sec-Q3} If $G=\Gal(\Q(E[3])/\Q)$ is abelian, then Theorem \ref{thm-diagonal} says that $G$ is diagonalizable in $\GL(2,\Z/3\Z)$. Thus, $G\cong \Z/2\Z$ or $(\Z/2\Z)^2$. 
\begin{enumerate}
\item If $G\cong \Z/2\Z$, then the existence of the Weil pairing forces $\Q(E[3])=\Q(\sqrt{-3})$. The curves $E$ with full $3$-torsion over $\Q(\sqrt{-3})$ are parametrized by $X(3)$, given by the Hesse cubic (see \cite{rubin}), which has a Weierstrass model:
$$E_{X(3)}(t):y^2=x^3-27t(t^3+8)x+54(t^6-20t^3-8),$$
with $j_{X(3)}(t)=j(E_{X(3)}(t))=\frac{27t^3(t^3+8)^3}{(t^3-1)^3}$. We note here that the curves in this family have a $3$-torsion point defined over $\Q$. A different model can be found in \cite{paladino}.
\item If $G\cong (\Z/2\Z)^2$, and $G$ is diagonalizable, then $G$ is a full split Cartan subgroup of $\GL(2,\Z/3\Z)$. The curves with split Cartan image are parametrized by the modular curve $X_s(3)$. Suppose $E/\Q$ has $G$ congruent to a full split Cartan subgroup of $\GL(2,\Z/3\Z)$. Then,
$$\rho_{E,3}(\GQ)\cong \left\{\left(\begin{array}{cc}
\phi_1(\sigma) & 0\\
0 & \phi_2(\sigma)
\end{array}\right): \sigma\in\GQ \right\},$$
and since $\phi_1\cdot \phi_2=\det(\rho_{E,3})= \chi_3$ is the $3$rd cyclotomic character, we must have $\phi_1=\chi_3\psi$ and $\phi_2=\psi^{-1}=\psi$, where $\psi\colon \GQ\to(\Z/3\Z)^\times$ is a quadratic character. Let $E^\psi/\Q$ be the quadratic twist of $E/\Q$ by $\psi$. Then,
$$\rho_{E^\psi,3}(\GQ)\cong \left\{\left(\begin{array}{cc}
\chi_3(\sigma) & 0\\
0 & 1 
\end{array}\right): \sigma\in\GQ \right\},$$
and, therefore, $E^\psi$ is a curve over $\Q$ with full $3$-torsion over $\Q(\sqrt{-3})$. Hence, $j(E^\psi)$ is one of the $j$-invariants that appear associated to a non-cuspidal rational point on $X(3)$. Hence, $j(E)=j_{X(3)}(t)=\frac{27t^3(t^3+8)^3}{(t^3-1)^3}$ for some $t\in\Q$. Conversely, if $j(E)=j(E_{X(3)}(t))$, then $E$ is a quadratic twist of a curve $E_{X(3)}(t)/\Q$ as above, with full $3$-torsion over $\Q(\sqrt{-3})$, and the image of $\rho_{E,3}$ is contained in the split Cartan. If the quadratic twist is by a character $\psi$ such that the quadratic character $\psi\chi_3$ is non-trivial, then $G=(\Z/2\Z)^2$, as desired. In particular, we have shown that $j_{X_s(3)}(t)=j_{X(3)}(t)$.
\end{enumerate}

\subsection{$\Q(E[4])$}\label{sec-Q4} If $G=\Gal(\Q(E[4])/\Q)$ is abelian, then Theorem \ref{thm-2torsab} shows that $G\cong (\Z/2\Z)^k$, for $k=1,2,3$, or $4$. Hence, $\Gal(\Q(E[4])/\Q)\cong\rho_{E,4}(\GQ)$ appears in the list of possible abelian images in the proof of Theorem \ref{thm-2torsab} (which in turn are the images classified in \cite{rouse}), as an abelian group in level $4$ or $8$, or as an abelian group of level $2$ that stays abelian under full inverse image in level $4$. Here we show examples of infinite families (most of them as they appear in \cite{rouse}) for each possible combination of Galois group $G$, and rational $4$-torsion $E(\Q)[4]$.
\begin{enumerate}
\item $G\cong \Z/2\Z$, and $E(\Q)[4]\cong \Z/4\Z$ (in this case, $\Q(E[4])=\Q(i)=\Q(\zeta_4)$):
$$\texttt{X60d}: y^2 = x^3 + (-432t^8 + 1512t^4 - 27)x + (3456t^{12}
    + 28512t^8 - 7128t^4 - 54).$$
    
    \item $G\cong \Z/2\Z$, and $E(\Q)[4]\cong \Z/2\Z\oplus \Z/4\Z$ (in this case, $\Q(E[4])=\Q(i)=\Q(\zeta_4)$):
    $$\texttt{X58i}: y^2 = x^3 + (-27t^8 - 378t^4 - 27)x + (54t^{12} -
        1782t^8 - 1782t^4 + 54).$$
\item $G\cong (\Z/2\Z)^2$, and $E(\Q)[4]\cong \Z/2\Z$:
$$\texttt{X27f}:  y^2 = x^3 + (-432t^6 + 1512t^4 - 27t^2)x +
    (3456t^9 + 28512t^7 - 7128t^5 - 54t^3).      $$
    Alternatively, the quadratic twists of $\texttt{X60d}$ are also in this category (which correspond to image $\texttt{X60}$ in \cite{rouse}).
\item $G\cong (\Z/2\Z)^2$, and $E(\Q)[4]\cong \Z/4\Z$:
$$\texttt{X27h}: y^2 = x^3 + (-432t^4 + 1512t^2 - 27)x + (3456t^6 +
    28512t^4 - 7128t^2 - 54).$$
\item $G\cong (\Z/2\Z)^2$, and $E(\Q)[4]\cong \Z/2\Z\oplus \Z/2\Z$: there are $5$ families of curves with these properties. One family is given by the quadratic twists of $\texttt{X58i}$ given above, which in fact constitute the points on the modular curve $X_s(4)$. Other families with these properties are $\texttt{X24d}$, $\texttt{X24e}$, $\texttt{X25h}$, and $\texttt{X25i}$. For instance, here is $\texttt{X24d}$:
 $$\texttt{X24d}:y^2=x^3 + (-27t^8 - 81t^6 - 108t^4 - 81t^2 -
     27)x + (54t^{12} + 243t^{10} + 324t^8 - 324t^4 - 243t^2 - 54).$$
\item $G\cong (\Z/2\Z)^2$, and $E(\Q)[4]\cong \Z/2\Z\oplus \Z/4\Z$:
 $$\texttt{X25n}: y^2 = x^3 + (-27t^4 + 27t^2 - 27)x + (54t^6 -
    81t^4 - 81t^2 + 54).$$
\item $G\cong (\Z/2\Z)^3$, and $E(\Q)[4]\cong \Z/2\Z$: the curves with this property are quadratic twists of $\texttt{X27f}$ above.
\item $G\cong (\Z/2\Z)^3$, and $E(\Q)[4]\cong \Z/2\Z\oplus \Z/2\Z$: there are four families of curves with these properties. Two of the families are the quadratic twists of $\texttt{X24d}$ and $\texttt{X25h}$, respectively. The other two families are $\texttt{X8b}$ and $\texttt{X8d}$. For instance, here is $\texttt{X8d}$:
$$\texttt{X8d}: y^2  = x^3 + (-108t^2 - 324)x + (432t^3 - 3888t).$$
\item $G\cong (\Z/2\Z)^4$, and $E(\Q)[4]\cong \Z/2\Z\oplus \Z/2\Z$: this is the family of quadratic twists of the curves $E$ in the family $\texttt{X8d}$, such that the quadratic field defined by the twist is not contained in the $(\Z/2\Z)^3$-extension defined by $\Q(E[4])/\Q$.
\end{enumerate}

\subsection{$\Q(E[5])$}\label{sec-Q5} If $G=\Gal(\Q(E[5])/\Q)$ is abelian, then Theorem \ref{thm-diagonal} says that $G$ is diagonalizable in $\GL(2,\Z/5\Z)$. Hence, $G\subseteq (\Z/5\Z)^\times \times (\Z/5\Z)^\times\cong (\Z/4\Z)^2$. Since $\Q(\zeta_5)\subseteq \Q(E[5])$, the only possibilities are $G\cong \Z/4\Z$, or $\Z/2\Z\times \Z/4\Z$, or $(\Z/4\Z)^2$.
\begin{enumerate}
\item If $G\cong \Z/4\Z$, then the existence of the Weil pairing forces $\Q(E[5])=\Q(\zeta_5)$. The curves $E$ with full $5$-torsion over $\Q(\zeta_5)$ are parametrized by $X(5)$, which has a Weierstrass model (see \cite{rubin}, who acknowledge \cite{klein}):
\small
$$E_{X(5)}(t):y^2=x^3-\frac{t^{20}-228t^{15}+494t^{10}+228t^5+1}{48}x+\frac{t^{30}+522t^{25}-10005t^{20}-10005t^{10}-522t^5+1}{864},$$
\normalsize with $j_{X(5)}(t)=j(E_{X(5)}(t))=\frac{-(t^{20}-228t^{15}+494t^{10}+228t^5+1)^3}{t^5(t^{10}+11t^5-1)^5}$. 

\item If $G\cong (\Z/2\Z)\times (\Z/4\Z)$, then a similar argument to the case of $G\cong (\Z/2\Z)^2$ for $\Q(E[3])$ shows that $E$ must be a quadratic twist of one of the curves $E_{X(5)}(t)$ given in the previous case (in fact, any non-trivial twist, except for the quadratic twist $E^5$ of $E$ by $d=5$). Hence, $j(E)=j(E_{X(5)}(t))$ for some $t\in\Q$.
\item If $G\cong (\Z/4\Z)^2$, and $G$ is diagonalizable, then $G$ is a full split Cartan of $\GL(2,\Z/5\Z)$. The curves with split Cartan image are parametrized by the modular curve $X_s(5)$. A model can be found in \cite{zywina1}, with $j$-line $j\colon X_s(5)\to \PP^1(\Q)$ given by 
$$j_{X_s(5)}(t)=\frac{(t^2+5t+5)^3(t^4+5t^2+25)^3(t^4+5t^3+20t^2+25t+25)^3}{(t(t^4+5t^3+15t^2+25t+25))^5}.$$
\end{enumerate} 

\subsection{$\Q(E[6])$}\label{sec-Q6} If $G=\Gal(\Q(E[6])/\Q)$ is abelian, then Lemma \ref{lem-6abelian} says that $\Gal(\Q(E[2])/\Q)\cong \Z/2\Z$ (and, therefore, there is a $2$-torsion point defined over $\Q$). If $\Q(E[3])/\Q$ is abelian, then we know that $\Gal(\Q(E[3])/\Q)\cong (\Z/2\Z)^k$ for $k=1$ or $2$. Hence, $G\cong (\Z/2\Z)^2$ or $(\Z/2\Z)^3$. Moreover, if $k=1$, then $E$ has a $3$-torsion point over $\Q$, and an independent rational $3$-isogeny, and if $k=2$, then $E$ has two independent rational $3$-isogenies (and $E$ is a quadratic twist of a curve $E'$ with $\Gal(\Q(E'[6])/\Q)\cong (\Z/2\Z)^2$) 
\begin{enumerate}
\item If $G=\Gal(\Q(E[6])/\Q)\cong (\Z/2\Z)^2$, then $E/\Q$ is characterized by having two distinct rational $3$-isogenies and a rational point of order $6$. In order to build the curves in this family, we first note that $E$ is isogenous to a curve with $E'$ with a $\Q$-isogeny of degree $18$. Hence, $j(E')=j_{X_0(18)}(t)$, where the $j$-line for $X_0(18)$ is given by (see \cite{ishii})
$$j_{X_0(18)}(t) = \frac{(t^3-2)^3(t^9-6t^6-12t^3-8)^3}{t^9(t^3-8)(t^3+1)^2}.$$
Now, using the formulas of V\'elu \cite{velu}, we find the formula for the $j$-invariant of the curve $E$ such that $E'\to E$ is an isogeny of degree $6$. The $j$-invariant of $E$, in terms of the parameter $t$, is given by
$$j(E_t)=\frac{(t^3 - 2)^3(t^3 + 6t - 2)^3(t^6 - 6t^4 - 4t^3 + 36t^2 + 12t + 4)^3}{(t - 2)^3t^3(t + 1)^6(t^2 - t + 1)^6(t^2 + 2t + 4)^3},$$
and a family of elliptic curves $E_t$ with $j$-invariant equal to $j(E_t)$ is given by
\begin{align*} E_t: y^2 & = x^3 + (-27t^{12} + 216t^9 - 6480t^6 + 12528t^3
    - 432)x \\
    &+ (54t^{18} - 648t^{15} - 25920t^{12} + 166320t^9 - 651888t^6 +
    222912t^3 + 3456).
    \end{align*}
Finally, one can verify that $E_t$ has a $\Q(t)$-rational point of order $6$, and two rational $3$-isogenies.
\item If $G=\Gal(\Q(E[6])/\Q)\cong (\Z/2\Z)^3$, then $\Gal(\Q(E[3])/\Q)\cong (\Z/2\Z)^2$ and by our results in Section \ref{sec-Q3} for $n=3$, we have that  $E/\Q$ is a quadratic twist of one of the elliptic curves $E_t$ with $6$-torsion defined over a biquadratic field $F$.  Conversely, if $E_t$ is as above, $F=\Q(E[6])$ is a biquadratic field, and if $K=\Q(\sqrt{d})$ is a quadratic field not contained in $F$, then $E_t^d$, the quadratic twist of $E_t$ by $d$, satisfies $\Gal(\Q(E_t^d[6])/\Q)\cong (\Z/2\Z)^3$.
\end{enumerate}

\subsection{$\Q(E[8])$}\label{sec-Q8} If $G=\Gal(\Q(E[8])/\Q)$ is abelian, then Theorem \ref{thm-2torsab} shows that $G\cong (\Z/2\Z)^k$, for $k=4,5$, or $6$. Hence, $\Gal(\Q(E[8])/\Q)\cong\rho_{E,8}(\GQ)$ appears in the list of possible abelian images in the proof of Theorem \ref{thm-2torsab} (which in turn are the images classified in \cite{rouse}), as an abelian group in level $8$, or as an abelian group of level $4$ that stays abelian under full inverse image in level $8$. Here we show examples of infinite families (most of them as they appear in \cite{rouse}) for each possible combination of Galois group $G$, and rational $8$-torsion $E(\Q)[8]$.
\begin{enumerate}
\item $G\cong (\Z/2\Z)^4$, and $E(\Q)[8]\cong \Z/2\Z\oplus\Z/2\Z$: there are $28$ families of curves with these properties (see Table 4 in section \ref{table4} for the full list and concrete examples). For instance,
\begin{align*} \texttt{X183a}: y^2 = x^3 & + (-108t^{24} - 216t^{20} - 1620t^{16} -
    3024t^{12} - 1620t^8 - 216t^4 - 108)x\\
    & + (432t^{36} + 1296t^{32} - 12960t^{28}
    - 42336t^{24} - 57024t^{20} - 57024t^{16}\\
    & - 42336t^{12} - 12960t^8 + 1296t^4 +
    432).
    \end{align*}
\item $G\cong (\Z/2\Z)^4$, and $E(\Q)[8]\cong \Z/2\Z\oplus\Z/4\Z$: there are $6$ families of curves with these properties (see Table 4 in section \ref{table4} for the full list and concrete examples). For instance,
\begin{align*} \texttt{X183d}: y^2 = x^3 + (-27t^{16} - 378t^8 - 27)x + (54t^{24} -
    1782t^{16} - 1782t^8 + 54).
\end{align*} 
\item $G\cong (\Z/2\Z)^5$, and $E(\Q)[8]\cong \Z/2\Z\oplus\Z/2\Z$: there are $13$ families of curves with these properties (see Table 4 in section \ref{table4} for the full list and concrete examples). For instance,
\begin{align*} \texttt{X58b}: y^2 = x^3 + (-108t^{10} - 1512t^6 - 108t^2)x +
    (432t^{15} - 14256t^{11} - 14256t^7 + 432t^3).
\end{align*} 
\item $G\cong (\Z/2\Z)^5$, and $E(\Q)[8]\cong \Z/2\Z\oplus\Z/4\Z$: these are the curves in the family $\texttt{X58i}$ that was already given in the section for abelian $\Q(E[4])/\Q$ extensions.
\item $G\cong (\Z/2\Z)^6$, and $E(\Q)[8]\cong \Z/2\Z\oplus\Z/2\Z$: the curves with these properties belong to the family of quadratic twists of the curves $E$ in the family $\texttt{X58i}$, such that the quadratic field defined by the twist is not contained in the $(\Z/2\Z)^5$-extension defined by $\Q(E[8])/\Q$.
\end{enumerate}

\section{Appendix: Tables}\label{sec-appendix}

In this section we include four tables that summarize our findings and provide concrete examples of elliptic curves (or families of elliptic curves) with each possible isomorphism type of abelian division field, and torsion structure over $\Q$.

\subsection{Table 1: CM curves} \label{table1} In this table, for each rational $j$-invariant with complex multiplication (and for each $\Q$-isomorphism class), we give 
\begin{enumerate}[(a)]
\item the degrees of the $\Q$-isogenies (with cyclic kernel) for elliptic curves in this $\Q$-isomorphism class, 
\item the Galois structure of the $2$-torsion division field $\Q(E[2])$, 
\item the largest division field $\Q(E[n])$ that is abelian over $\Q$, and
\item the largest division field $\Q(E[n])$ that is isomorphic to $\Q(\zeta_n)$, if any. 
\end{enumerate} 

\vskip 0.1in

\begin{center} 
\renewcommand{\arraystretch}{1.2}
\begin{tabular}{|c|c|c|c|c|c|c|}
\hline
$d_K$ & \multicolumn{2}{c|}{$j$} & Isogenies & $\Q(E[2])/\Q$ & $\Q(E[n])$ abelian & $\Q(\zeta_n)$\\
\hline
\hline
\multirow{5}{*}{$-3$}  & \multirow{3}{*}{$0$} & $y^2=x^3+t^3$ & $1,2,3,6$ & $\Q(\sqrt{-3})$ & $\Q(E[2])=\Q(\sqrt{-3})$ & none\\
& &$y^2=x^3+16t^3$ & $1,3,3,9$ & $S_3$ & $\!\Q(E[3])=\Q(\sqrt{-3},\sqrt{t})\!\!$ & $\!\Q(E[3])\!\!$\\
  & & $\!y^2=x^3+s,\,s\neq t^3,16t^3\!\!$ & $1,3$ & $S_3$ & none & none\\
  \cline{2-7}
  & \multicolumn{2}{c|}{$2^4\cdot 3^3\cdot 5^3$} & $1,2,3,6$ & $\Q(\sqrt{3})$ & $\Q(E[2])=\Q(\sqrt{3})$ & none\\
& \multicolumn{2}{c|}{$-2^{15}\cdot 3\cdot 5^3$} & $1,3,9,27$ & $S_3$ & none & none\\
\hline
\multirow{4}{*}{$-4$} & \multirow{3}{*}{$2^6\cdot 3^3$} & $y^2=x^3+t^2x$ & $1,2,4,4$ & $\Q(i)$ & $\Q(E[4])=\Q(\zeta_8,\sqrt{t})$ & none\\
& & $y^2=x^3-t^2x$ & $1,2,2,2$ & $\Q$ & $\Q(E[4])=\Q(\zeta_8,\sqrt{t})$ & $\Q$\\
& &  $y^2=x^3+sx$, $s\neq \pm t^2$ & $1,2$ & $\Q(\sqrt{-s})$ & $\Q(E[2])=\Q(\sqrt{-s})$ & none\\
\cline{2-7}
  & \multicolumn{2}{c|}{$2^3\cdot 3^3\cdot 11^3$} & $1,2,4,4$ & $\Q(\sqrt{2})$ & $\Q(E[2])=\Q(\sqrt{2})$ & none\\
\hline
\multirow{2}{*}{$-7$} & \multicolumn{2}{c|}{$-3^3\cdot 5^3$} & $1,2,7,14$ & $\Q(\sqrt{-7})$ & $\Q(E[2])=\Q(\sqrt{-7})$ & none\\
  & \multicolumn{2}{c|}{$3^3\cdot 5^3\cdot 17^3$} & $1,2,7,14$ & $\Q(\sqrt{7})$ & $\Q(E[2])=\Q(\sqrt{7})$ & none\\
\hline
$-8$ & \multicolumn{2}{c|}{$2^6\cdot 5^3$} & $1,2$ & $\Q(\sqrt{2})$ & $\Q(E[2])=\Q(\sqrt{2})$ & none\\
\hline
$-11$ & \multicolumn{2}{c|}{$-2^{15}$} & $1,11$ & $S_3$ & none & none\\
\hline
$-19$ & \multicolumn{2}{c|}{$-2^{15}\cdot 3^3$} & $1,19$ & $S_3$ & none & none\\
\hline
$-43$ & \multicolumn{2}{c|}{$-2^{18}\cdot 3^3\cdot 5^3$} & $1,43$ & $S_3$ & none & none\\
\hline
$-67$ & \multicolumn{2}{c|}{$-2^{15}\cdot 3^3\cdot 5^3\cdot 11^3$} & $1,67$ & $S_3$ & none & none\\
\hline
$-163$ & \multicolumn{2}{c|}{$-2^{18}\cdot 3^3\cdot 5^3\cdot 23^3\cdot 29^3$} & $1,163$ & $S_3$ & none & none\\
\hline
\end{tabular}

\end{center} 
\newpage

\subsection{Table 2: $n=3,5$, or $6$} \label{table2} In this table, for each $n=3$, $5$, or $6$, and for each of the possible isomorphism type of $\Gal(\Q(E[n])/\Q)$, we give 
\begin{enumerate}[(a)]
\item the possible torsion structures of $E(\Q)[n]$, 
\item an example of an elliptic curve $E/\Q$ with each Galois group and torsion structure, using the notation of Cremona's database \cite{cremonaweb}, and 
\item the actual division field $\Q(E[n])$ for the concrete example curve given in (b). We remark here that the actual field $\Q(E[n])$ depends on the choice of $E/\Q$, when $\Q(\zeta_n)\subsetneq \Q(E[n])$. 
\end{enumerate}

\vskip 0.1in

\begin{center}
\renewcommand{\arraystretch}{1.3}
\begin{tabular}{|c|c|c|c|c|}
\hline
$n$ & $\Gal(\Q(E[n])/\Q)$ & $E(\Q)[n]$ & Cremona & $\Q(E[n])$\\
\hline\hline
\multirow{2}{*}{$3$} &$\Z/2\Z$  & $\Z/3\Z$ & \texttt{19a1} & $\Q(\zeta_3)$\\
\cline{2-5}
                              & \multirow{1}{*}{$(\Z/2\Z)^2$}& $\{\mathcal O\}$ &    \texttt{175b2} & $\Q(\zeta_3,\sqrt{5})$\\
                              \hline\hline
\multirow{5}{*}{$5$} &\multirow{2}{*}{$\Z/4\Z$}  &  $\{\mathcal O\}$&   \texttt{275b2}&\multirow{2}{*}{$\Q(\zeta_5)$}  \\
\cline{3-4}
&  & $\Z/5\Z$  & \texttt{11a1} &  \\
\cline{2-5}

                              & $\Z/2\Z\oplus\Z/4\Z$& $\{\mathcal O \}$ & \texttt{704a2} &  $\Q(\zeta_5,\sqrt{2})$ \\
\cline{2-5}                           
                              & \multirow{2}{*}{$(\Z/4\Z)^2$}& \multirow{2}{*}{$\{\mathcal O \}$}  &  \multirow{2}{*}{\texttt{18176b2}} & $\Q(\zeta_5,\alpha)$ \\
                              & &  & &          $\mbox{\small$\alpha^4-4\alpha^2+2=0$}$\\                       
\hline\hline
\multirow{3}{*}{$6$} & \multirow{2}{*}{$(\Z/2\Z)^2$}  & $\Z/2\Z$ &    \texttt{98a3} &  \multirow{2}{*}{$\Q(\zeta_6,\sqrt{-7})$}  \\
\cline{3-4}
&  & $\Z/6\Z$ &   \texttt{14a1} &   \\
\cline{2-5}

                              & $(\Z/2\Z)^3$& $\Z/2\Z$ &   \texttt{448c3} & $\Q(\zeta_6,\sqrt{2},\sqrt{-7})$ \\
\hline
\end{tabular}
\end{center}

\newpage

\subsection{Table 3: $n=2$ or $4$}\label{table3} In this table, and Table 4 in Section \ref{table4}, for each $n=2$, $4$, or $8$, and for each of the possible isomorphism types of $\Gal(\Q(E[n])/\Q)$, we give
\begin{enumerate}[(a)]
\item the possible torsion structures of $E(\Q)[n]$, 
\item an infinite family of examples with each Galois group and torsion structure, using the notation of \cite{rouse},
\item an example of an elliptic curve in the family given in (b), and 
\item the actual division field $\Q(E[n])$ for the concrete example curve given in (c). We remark here that the actual field $\Q(E[n])$ depends on the choice of $E/\Q$, when $\Q(\zeta_n)\subsetneq \Q(E[n])$. 
\end{enumerate}

\vskip 0.1in

\begin{center}
\renewcommand{\arraystretch}{1.3}
\begin{tabular}{|c|c|c|c|c|c|}
\hline
$n$ & $\Gal(\Q(E[n])/\Q)$ & $E(\Q)[n]$ & R-ZB & Cremona & $\Q(E[n])$\\
\hline\hline
\multirow{4}{*}{$2$} & $\{0\}$ & $\Z/2\Z\oplus\Z/2\Z$ & \texttt{X8} & \texttt{315b2}& $\Q$  \\
\cline{2-6}
                              & $\Z/2\Z$& $\Z/2\Z$ & \texttt{X6} &\texttt{69a1}   & $\Q(\sqrt{-23})$ \\
\cline{2-6}                           
                              & \multirow{2}{*}{$\Z/3\Z$}& \multirow{2}{*}{\{$\mathcal O$\}} & \multirow{2}{*}{\texttt{X2}} & \multirow{2}{*}{\texttt{196a1}}& $\Q(\alpha)$ \\                   
                              & & & & & $\mbox{\small$\alpha^3 - \alpha^2 - 2\alpha + 1=0$}$\\            
\hline\hline
\multirow{17}{*}{$4$} &\multirow{2}{*}{$\Z/2\Z$}  & $\Z/4\Z$ &\texttt{X60d}  &   \texttt{40a4} & $\Q(\zeta_4)$\\
\cline{3-6}
                              & & $\Z/2\Z\oplus\Z/4\Z$ &  \texttt{X58i} & \texttt{195a3} & $\Q(\zeta_4)$\\
\cline{2-6} 
                              & \multirow{9}{*}{$(\Z/2\Z)^2$}&\multirow{2}{*}{$\Z/2\Z$}  & \texttt{X60} & \texttt{360e4} & $\Q(\zeta_4,\sqrt{3})$ \\
\cline{4-6}                           
                              & & &  \texttt{X27f} & \texttt{936i4} & $\Q(\zeta_4,\sqrt{3})$ \\
\cline{3-6}   
                      & & $\Z/4\Z$ &  \texttt{X27h} & \texttt{205a4} & $\Q(\zeta_4,\sqrt{10})$ \\
\cline{3-6}     
                      & &\multirow{5}{*}{$\Z/2\Z\oplus\Z/2\Z$}  &\texttt{X58} & \texttt{6435f3} & $\Q(\zeta_4,\sqrt{3})$ \\
\cline{4-6}     
                              & &  &  \texttt{X24d} & \texttt{2200f2} & $\Q(\zeta_4,\sqrt{5})$ \\
\cline{4-6}     
                              & &  &  \texttt{X24e} & \texttt{205a2} & $\Q(\zeta_4,\sqrt{41})$ \\
\cline{4-6}     
                              & &  &  \texttt{X25h} & \texttt{231a3} & $\Q(\zeta_4,\sqrt{7})$ \\
\cline{4-6}     
                              & &  &  \texttt{X25i} & \texttt{1287e4} & $\Q(\zeta_4,\sqrt{3})$ \\
\cline{3-6}     
                              & & $\Z/2\Z\oplus\Z/4\Z$ &  \texttt{X25n} & \texttt{231a2} & $\Q(\zeta_4,\sqrt{33})$ \\
\cline{2-6}                           
                              & \multirow{5}{*}{$(\Z/2\Z)^3$} & $\Z/2\Z$ &\texttt{X27}  & \texttt{1845c4}  & $\Q(\zeta_4,\sqrt{3},\sqrt{10})$\\                               
\cline{3-6}                           
& & \multirow{4}{*}{$\Z/2\Z\oplus\Z/2\Z$} &\texttt{X24} & \texttt{1845c2}  & $\Q(\zeta_4,\sqrt{3},\sqrt{41})$\\                               
\cline{4-6}                           
& &  & \texttt{X25} & \texttt{3465e2}  & $\Q(\zeta_4,\sqrt{3},\sqrt{35})$\\                                         
\cline{4-6}      
& &  & \texttt{X8b} & \texttt{1089g2}  & $\Q(\zeta_4,\sqrt{3},\sqrt{11})$\\                                         
\cline{4-6}      
& &  & \texttt{X8d} & \texttt{33a1}  & $\Q(\zeta_4,\sqrt{3},\sqrt{11})$\\                                         
\cline{2-6}      
                              & $(\Z/2\Z)^4$& $\Z/2\Z\oplus\Z/2\Z$ & \texttt{X8} & \texttt{315b2}  & $\Q(\zeta_4,\sqrt{3},\sqrt{5},\sqrt{7})$\\      
\hline
\end{tabular}
\end{center}
\newpage
\subsection{Table 4: $n=8$}\label{table4}
\begin{center}
\renewcommand{\arraystretch}{1.2}
\begin{tabular}{|c|c|c|c|c|c|}
\hline
$n$ & $\Gal(\Q(E[n])/\Q)$ & $E(\Q)[n]$ & R-ZB & Cremona & $\Q(E[n])$\\
\hline\hline

\multirow{34}{*}{$8$} &\multirow{34}{*}{$(\Z/2\Z)^4$}  & \multirow{28}{*}{$\Z/2\Z\oplus\Z/2\Z$} & X183a & 277440dv4 & 
\multirow{7}{*}{ $\Q(\zeta_8,\sqrt{15},\sqrt{17})$}\\  \cline{4-5}
& &  & \texttt{X183c} &  \texttt{130050bu3} & \\  \cline{4-5}
& &  & \texttt{X183e} &  \texttt{16320bb3} & \\ \cline{4-5}
& &  & \texttt{X183h} &  \texttt{8670v3} &\\  \cline{4-5}
& &  & \texttt{X183k} &  \texttt{16320bt4} & \\  \cline{4-5}
& &  & \texttt{X183l} &  \texttt{69360cb3} & \\  \cline{4-6}

& &  & \texttt{X187a} &  \texttt{4800b3} & \multirow{10}{*}{ $\Q(\zeta_8,\sqrt{3},\sqrt{5})$}\\  \cline{4-5}
& &  & \texttt{X187b} &  \texttt{2880r4} & \\  \cline{4-5}
& &  & \texttt{X187c} &  \texttt{225c4} & \\  \cline{4-5}
& &  & \texttt{X187e} &  \texttt{960g3} & \\  \cline{4-5}
& &  & \texttt{X187f} &  \texttt{14400y4} & \\  \cline{4-5}
& &  & \texttt{X187g} &  \texttt{45a4} & \\  \cline{4-5}
& &  & \texttt{X187h} &  \texttt{75b3} & \\  \cline{4-5}
& &  & \texttt{X187i} &  \texttt{960i4} & \\  \cline{4-5}
& &  & \texttt{X187j} &  \texttt{1200j4} & \\  \cline{4-5}
& &  & \texttt{X187l} &  \texttt{4800cd4} & \\  \cline{4-6}

& &  & \texttt{X189a} &  \texttt{141120el3} & \multirow{8}{*}{ $\Q(\zeta_8,\sqrt{3},\sqrt{7})$}\\  \cline{4-5}
& &  & \texttt{X189b} &  \texttt{20160ce4} &\\  \cline{4-5}
& &  & \texttt{X189c} &  \texttt{1470k4} &\\  \cline{4-5}
& &  & \texttt{X189h} &  \texttt{47040dm3} &\\  \cline{4-5}
& &  & \texttt{X189i} &  \texttt{6720cd4} &\\  \cline{4-5}
& &  & \texttt{X189j} &  \texttt{4410r3} &\\  \cline{4-5}
& &  & \texttt{X189k} &  \texttt{630c4} & \\  \cline{4-5}
& &  & \texttt{X189l} &  \texttt{6720c3} &\\  \cline{4-6}

& &  & \texttt{X183b} &  \texttt{196800by3} & \multirow{4}{*}{ $\Q(\zeta_8,\sqrt{5},\sqrt{41})$}\\  \cline{4-5}
& &  & \texttt{X183f} &  \texttt{252150c4} & \\  \cline{4-5}
& &  & \texttt{X183g} &  \texttt{6150n3} &\\  \cline{4-5}
& &  & \texttt{X183j} &  \texttt{50430z4} &\\  \cline{3-6}
                              &  & \multirow{6}{*}{$\Z/2\Z\oplus\Z/4\Z$} &   \texttt{X183d} &  \texttt{510e3} &$\Q(\zeta_8,\sqrt{15},\sqrt{17})$\\  \cline{4-6}
& &  & \texttt{X183i} &  \texttt{1230f3} & $\Q(\zeta_8,\sqrt{5},\sqrt{41})$\\  \cline{4-6}

& &  & \texttt{X187d} &  \texttt{15a1} & \multirow{2}{*}{ $\Q(\zeta_8,\sqrt{3},\sqrt{5})$}\\  \cline{4-5}
& &  & \texttt{X187k} &  \texttt{240d4} & \\  \cline{4-6}

& &  & \texttt{X189d} &  \texttt{210e3} & \multirow{2}{*}{ $\Q(\zeta_8,\sqrt{3},\sqrt{7})$}\\  \cline{4-5}
& &  & \texttt{X189e} &  \texttt{1680p4} & \\  \cline{2-6}
\hline
\end{tabular}
\end{center}
\newpage 
 \begin{center}
\renewcommand{\arraystretch}{1.3}
\begin{tabular}{|c|c|c|c|c|c|}
\hline
$n$ & $\Gal(\Q(E[n])/\Q)$ & $E(\Q)[n]$ & R-ZB & Cremona & $\Q(E[n])$\\
\hline\hline
\multirow{15}{*}{$8$}                      & \multirow{14}{*}{$(\Z/2\Z)^5$}  & \multirow{13}{*}{$\Z/2\Z\oplus\Z/2\Z$} &\texttt{X183} & \texttt{1530c4} & $\Q(\zeta_8,\sqrt{3},\sqrt{5},\sqrt{17})$\\  \cline{4-6}
& &  & \texttt{X187} & \texttt{735e4} &$\Q(\zeta_8,\sqrt{3},\sqrt{5},\sqrt{7})$\\  \cline{4-6}
& &  & \texttt{X189} & \texttt{25410bj3} & $\Q(\zeta_8,\sqrt{3},\sqrt{7},\sqrt{11})$\\  \cline{4-6}

& &  & \texttt{X58a} & \texttt{2535f3} &  \multirow{11}{*}{$\Q(\zeta_8,\sqrt{3},\sqrt{5},\sqrt{13})$}\\  \cline{4-5}
& &  & \texttt{X58b} & \texttt{585f3} & \\  \cline{4-5}
& &  & \texttt{X58c} & \texttt{7605p3} &\\  \cline{4-5}
& &  & \texttt{X58d }& \texttt{2925g3} &\\  \cline{4-5}
& &  & \texttt{X58e} & \texttt{12480a3} &\\  \cline{4-5}
& &  & \texttt{X58f} & \texttt{486720dr3} &\\  \cline{4-5}
& &  & \texttt{X58g} & \texttt{38025bb4} &\\  \cline{4-5}
& &  & \texttt{X58h} & \texttt{40560bg4} &\\  \cline{4-5}
& &  & \texttt{X58j} & \texttt{187200eh4} &\\  \cline{4-5}
& &  & \texttt{X58k} & \texttt{975a3} & \\  \cline{3-6}
  & & $\Z/2\Z\oplus\Z/4\Z$& \texttt{X58i} & \texttt{195a3} &$\Q(\zeta_8,\sqrt{3},\sqrt{5},\sqrt{13})$\\  \cline{4-6}
\cline{2-6}             
                              & $(\Z/2\Z)^6$& $\Z/2\Z\oplus\Z/2\Z$ & \texttt{X58}  & \texttt{6435f3}  & $\Q(\zeta_8,\sqrt{3},\sqrt{5},\sqrt{11},\sqrt{13})$\\                               
\hline
\end{tabular}

\end{center}


\begin{thebibliography}{9}
%



\bibitem{bilu} Y. Bilu, P. Parent, {\em Serre's uniformity problem in the split Cartan case}, Annals of Mathematics, Volume 173 (2011), Issue 1, pp. 569-584.

\bibitem{lnm476} B. J. Birch, W. Kuyk (Editors), {\it Modular functions of one variable IV}, Lecture Notes in Mathematics 476, Berlin-Heidelberg-New York, Springer 1975.

\bibitem{magma} W. Bosma, J. Cannon, and C. Playoust, {\it The Magma algebra system. I. The user language}, J. Symbolic Comput., 24 (1997), 235-265.

\bibitem{alina1} A. C. Cojocaru, {\em On the surjectivity of the Galois representations associated to non-CM elliptic curves (with
an appendix by Ernst Kani)}, Canad. Math. Bull. 48 (2005), pp. 16-31.

\bibitem{alina2} A. C. Cojocaru, C. Hall, {\em Uniform results for Serre's theorem for elliptic curves}, Int. Math. Res. Not. 2005,
pp. 3065-3080.



\bibitem{cremonaweb} J. E. Cremona, {\it Elliptic curve data for conductors up to 500.000.} Available at
  \href{http://homepages.warwick.ac.uk/~masgaj/ftp/data/}{http://homepages.warwick.ac.uk/$\sim$masgaj/ftp/data/}, 2019.

\bibitem{daniels} H. Daniels, {\em Siegel functions, modular curves, and Serre's uniformity problem}, Albanian J. of Math, Vol. 9, Number 1 (2015), pp. 3-29.

\bibitem{dok} T. Dokchitser, V. Dokchitser, {\em Surjectivity of mod $2^n$ representations of elliptic curves}, Math. Z., Vol. 272, Issue 3-4 (2012), pp. 961-964.

\bibitem{elkies1} N. Elkies, {\it Elliptic and modular curves over finite fields and related computational issues}, in Computational Perspectives on Number Theory: Proceedings of a Conference in Honor of A.O.L. Atkin (D.A. Buell and J.T. Teitelbaum, eds.; AMS/International Press, 1998), pp. 21-76.

\bibitem{elkies2} N. Elkies, {\it Explicit Modular Towers}, in Proceedings of the Thirty-Fifth Annual Allerton Conference on Communication, Control and Computing (1997, T. Basar, A. Vardy, eds.), Univ. of Illinois at Urbana-Champaign 1998, pp. 23-32  (math.NT/0103107 on the arXiv).

\bibitem{kleinfricke} R. Fricke, F. Klein, {\it Vorlesungen über die Theorie der elliptischen Modulfunctionen} (Volumes 1 and 2), B. G. Teubner, Leipzig 1890, 1892.

\bibitem{fricke} R. Fricke, {\it Die elliptischen Funktionen und ihre Anwendungen}. Leipzig-Berlin: Teubner 1922.

\bibitem{fujita} Y. Fujita, {\em Torsion subgroups of elliptic curves in elementary abelian 2-extensions of $\Q$}, J. Number Theory 114 (2005), 124-134.

\bibitem{ishii} N. Ishii, {\it Rational Expression for J-invariant Function
in Terms of Generators of Modular Function Fields}, International Mathematical Forum, 2, 2007, no. 38, pp. 1877 - 1894.

\bibitem{kami} S. Kamienny, {\em  Torsion points on elliptic curves and q-coefficients of modular forms}, Invent. Math. 109 (1992), pp. 129-133.

\bibitem{kenku} M. A. Kenku, {\em On the number of $\Q$-isomorphism classes of elliptic curves in each $\Q$-isogeny class}, J. Number Th. 15 (1982), 199-202.

\bibitem{kenku2} M. A. Kenku, F. Momose, {\em Torsion points on elliptic curves defined over quadratic fields}, Nagoya Math. J. 109 (1988), pp. 125-149.

\bibitem{klein} F. Klein, {\it Lectures on the icosahedron and the solution of equations of the fifth degree}, London (1913).

\bibitem{knapp} A. W. Knapp, {\it Elliptic curves}. Mathematical Notes, 40. Princeton University Press (1992).

\bibitem{kraus} A. Kraus, {\em Une remarque sur les points de torsion des courbes elliptiques}, C. R. Acad. Sci. Paris Sr. I Math. 321
(1995), pp. 1143-1146.

\bibitem{laska} M. Laska and M. Lorenz, {\it  Rational points on elliptic curves over $\Q$ in elementary
abelian 2-extensions of $\Q$}, J. Reine Angew. Math. 355 (1985), 163-172.

\bibitem{lozano0} \'A. Lozano-Robledo, {\it On the field of definition of $p$-torsion points on elliptic curves over the rationals}, in the Mathematische Annalen, Vol 357, Issue 1 (2013), pp. 279-305.

\bibitem{lozano1} \'A. Lozano-Robledo, {\em Division fields of elliptic curves with minimal ramification}, preprint (to appear in Revista Matem\'atica Iberoamericana).

\bibitem{maier} R. Maier, {\it On Rationally Parametrized Modular Equations}, J. Ramanujan Math. Soc. 24 (2009), pp. 1 - 73.

\bibitem{masser} D. W. Masser, G. W\"ustholz, {\em Galois properties of division fields of elliptic curves}, Bull. London Math. Soc. 25
(1993), pp. 247-254.

\bibitem{mazur1} B. Mazur, {\it Rational isogenies of prime degree}, Inventiones Math. 44 (1978), pp. 129 - 162.

\bibitem{merel} L. Merel, {\em Sur la nature non-cyclotomique des points d'ordre fini des courbes elliptiques (Appendice de E. Kowalski et P. Michel)}, Duke Mathematical Journal, Vol. 110 (2001), No. 1, pp. 81-119.



\bibitem{merel2} L. Merel, W. Stein, {\em The field generated by the points of small prime order on an elliptic curve}, IMRN (2001), No. 20, pp. 1075-1082.

\bibitem{najman} F. Najman, {\it Torsion of rational elliptic curves over cubic fields and sporadic points on $X_1(n)$}, Math. Res. Lett., to appear.

\bibitem{paladino} L. Paladino, {\em Elliptic curves with $\Q(E[3])=\Q(\zeta_3)$ and counterexamples to local-global divisibility by $9$}, Journal de Th\'eorie des Nombres de Bordeaux 22 (2010), pp. 139-160.

\bibitem{pellarin} F. Pellarin, {\em Sur une majoration explicite pour un degr\'e d'isog\'enie liant deux courbes elliptiques}, Acta Arith.
100 (2001), pp. 203-243.


\bibitem{rebolledo} M. Rebolledo Hauchart, {\em Corps engendr\'e par les points de 13-torsion des courbes elliptiques}, Acta Arithmetica, 109.3 (2003), pp. 219-230.

\bibitem{rouse} J. Rouse, D. Zureick-Brown, {\em Elliptic curves over $\Q$ and $2$-adic images of Galois}, Research in Number Theory 1:12, 2015. (Data files and subgroup descriptions available at \href{http://users.wfu.edu/rouseja/2adic/}{http://users.wfu.edu/rouseja/2adic/}).

\bibitem{rubin} K. Rubin and A. Silverberg, {\it Families of elliptic curves with constant mod $p$ representations}, Elliptic curves, modular forms, and Fermat's last theorem (Hong Kong, 1993), Ser. Number Theory, I, Int. Press, Cambridge, MA, 1995, pp. 148-161.

\bibitem{serre1} J-P. Serre, {\em Propri\'et\'es galoisiennes des points d'ordre fini des courbes elliptiques}, Invent. Math. 15 (1972), pp. 259-331.

\bibitem{shanks} D. Shanks, {\em The Simplest Cubic Fields}, Mathematics of Computation, Vol. 28, No. 128. (Oct., 1974), pp. 1137-1152.

\bibitem{silverman} J. H. Silverman, {\em The Arithmetic of Elliptic Curves}, Springer-Verlag, 2nd Edition, New York, 2009.

\bibitem{velu} J. V\'elu, {\it Isog\'enies entre courbes elliptiques}, C.R. Acad. Sc. Paris, S\'erie A., 273, pp. 238-241 (1971).

\bibitem{zywina1} D. Zywina, {\em On the possible images of the mod $\ell$ representations associated to elliptic curves over $\Q$}. (\href{http://arxiv.org/abs/1508.07660}{arXiv:1508.07660}).

\end{thebibliography}
\end{document}